\documentclass[11pt]{article}

\usepackage[margin=2cm]{geometry}
\usepackage{amssymb,latexsym,amsmath,amsthm}
\usepackage[all]{xy}

\newtheorem{proposition}{Proposition}[section]
\newtheorem{theorem}[proposition]{Theorem}
\newtheorem{lemma}[proposition]{Lemma}
\newtheorem{corollary}[proposition]{Corollary}
\newtheorem{definition}[proposition]{Definition}
\theoremstyle{remark}
\newtheorem{remark}[proposition]{Remark}

\newcommand{\hh}[1]{\widehat{#1}}
\newcommand{\G}{\mathbb{G}}
\newcommand{\K}{\mathbb{K}}
\newcommand{\tp}{\xymatrix{*+<.7ex>[o][F-]{\scriptstyle\top}}}

\newcommand{\inv}{\operatorname{Inv}}

\newcommand{\mc}{\mathcal}
\newcommand{\mf}{\mathfrak}
\newcommand{\ip}[2]{\langle #1,#2 \rangle}

\newcommand{\vnten}{\overline\otimes}
\newcommand{\proten}{\widehat\otimes}
\newcommand{\ehten}{\otimes^{eh}}
\newcommand{\id}{\operatorname{id}}
\newcommand{\lin}{\operatorname{lin}}
\newcommand{\wap}{{\operatorname{wap}}}
\newcommand{\ap}{{\operatorname{ap}}}
\newcommand{\aone}{\Box}
\newcommand{\atwo}{\Diamond}

\newcommand{\sch}{\textsf{SCH}}
\newcommand{\wch}{\textsf{WCH}}
\newcommand{\lcg}{\textsf{LCG}}
\newcommand{\sap}{{\operatorname{SAP}}}

\begin{document}

\title{Quantum Eberlein compactifications and invariant means}
\author{Biswarup Das, Matthew Daws}
\maketitle

\begin{abstract}
We propose a definition of a ``$C^*$-Eberlein'' algebra, which is a weak form
of a $C^*$-bialgebra with a sort of ``unitary generator''.  Our definition is
motivated to ensure that commutative examples arise exactly from semigroups of
contractions on a Hilbert space, as extensively studied recently by Spronk and
Stokke.  The terminology arises as the Eberlein algebra, the uniform closure
of the Fourier-Stieltjes algebra $B(G)$, has character space $G^{\mathcal E}$, which
is the semigroup compactification given by considering all semigroups of
contractions on a Hilbert space which contain a dense homomorphic image of $G$.
We carry out a similar construction for locally compact quantum groups, leading
to a maximal $C^*$-Eberlein compactification.  We show that $C^*$-Eberlein
algebras always admit invariant means, and we apply this to prove various
``splitting'' results, showing how the $C^*$-Eberlein compactification splits
as the quantum Bohr compactification and elements which are annihilated by
the mean.  This holds for matrix coefficients, but for Kac algebras, we show it
also holds at the algebra level, generalising (in a semigroup-free way)
results of Godement.


\end{abstract}

\section{Introduction}

For a locally compact group $G$, there is a tight connection between semigroup
compactifications of $G$ and certain introverted subspaces of $C^b(G)$.  For
example, the weakly almost periodic functions on $G$ form a commutative
$C^*$-algebra $\wap(G)$ whose character space $G^\wap$ becomes a compact semigroup
whose product is jointly continuous.  Indeed, $G^\wap$ is the largest such
``semitopological'' semigroup which contains a dense homomorphic image of $G$.
For further details see the monograph \cite{bjm}.  It can be hard to describe
functions in $\wap(G)$, but any coefficient of a unitary representation,
that is, a function in the Fourier-Stieltjes algebra $B(G)$, is weakly almost
periodic, and so the uniform norm closure of $B(G)$ in $C^b(G)$, the Eberlein
algebra $\mc{E}(G)$, is a subalgebra of $\wap(G)$.  Recent work of
Spronk and Stokke, \cite{ss}, studied $\mc{E}(G)$ and its character space
$G^{\mc E}$ in detail.  In particular, $G^{\mc E}$ is the maximal semigroup which
arises as a semigroup of contractions on a Hilbert space; see also \cite{m}.

In this paper, we are interested in ``non-commutative'' Eberlein algebras.
There has recently been much interest in ``topological quantum groups'',
which we shall take to mean $C^*$-bialgebras, probably with additional structure,
such as (locally) compact quantum groups.
For example, if $G$ is a locally compact group, we consider the algebra
$C_0(G)$ with the coproduct $\Delta:C_0(G)\rightarrow M(C_0(G)\otimes C_0(G))
\cong C^b(G\times G); \Delta(f)(s,t) = f(st)$ which captures the group product.
The theory of locally compact quantum groups, \cite{kv}, studies such objects
where also left and right invariant weights are assumed.  Then those quantum
groups which arise from a commutative $C^*$-algebra are exactly those coming
from locally compact groups in this way.

In this paper, we study an analogous ``non-commutative'' model, which we call
a ``$C^*$-Eberlein algebra'', which is a $C^*$-algebra with additional structure.
Our construction ensures that if we have a commutative $C^*$-algebra then we
exactly recover a semigroup of contractions on a Hilbert space.  We then go on
to study compactifications: given a $C^*$-bialgebra (maybe with extra structure)
what is the ``maximal'' $C^*$-Eberlein algebra which contains a dense 
``homomorphic'' copy of the starting bialgebra?
We make sense of this, and prove the existence of such an object.

Much of our motivation comes from a longer-term project of seeing what
weakly-almost periodic functions (or uniformly continuous, and so forth)
might mean in the non-commutative setting.  For a locally compact group $G$
one can form the Banach algebra $L^1(G)$ and then recover $\wap(G)$ by looking
at the weakly almost periodic functionals on $L^1(G)$.  This Banach algebraic
definition makes sense in general, and has been studied for quantum groups,
see for example \cite[Section~4]{runde2} and \cite{hnr}.
However, to our knowledge there has been little study in terms of
compactifications (but compare \cite{daws, salmi}).  Of course, a first step in
this programme must be to have an independent model of what a ``semitopological
quantum semigroup'' is.  This paper represents a step in this direction, where
we can make use of Hilbert space techniques.

A major use of $\wap(G)$ is that it always admits an invariant mean (whereas
for $L^\infty(G)$ we of course need $G$ to be amenable).  For example, this
allows us to talk about ``mixing'' for actions of artibrary groups, not just
$\mathbb Z$ or $\mathbb R$, see \cite{br}.  The construction of this mean is
typically derived from a fixed-point theorem (compare
\cite[Chapter~4.2]{bjm}) but for $\mc{E}(G)$ one can instead use Hilbert space
techniques.  This idea also successfully works for non-commutative $C^*$-Eberlein
algebras, allowing us to construct invariant means.  As an application of this,
for Kac algebras (locally compact quantum groups with a particularly
well-behaved antipode) we show that the Eberlein compactification splits as
a direct sum of the almost periodic compactification and those ``functions''
which are annihilated by the mean, thus generalising Godemont's Theorem
\cite[Theorem~16]{G} in the commutative situation.

The paper is organised as follows: we first introduced some terminology and
notation.  In Section~\ref{sec:semigps} we study semigroups of contractions
on a Hilbert space, and show how to think of such objects in a $C^*$-algebra
framework.  We can then allow the $C^*$-algebra to be noncommutative, while
still maintaining a lot of interesting structure.  For example, we study
morphisms of such objects in Section~\ref{sec:semigps:morphisms}, in
Section~\ref{sec:baissues}
we use some dual Banach algebra techniques to study permanence properties
(quotients and subspaces) of these objects, and in
Section~\ref{sec:embedded} we study which such semigroups arise naturally
from a (locally compact) quantum group.  In Section~\ref{sec:means} we
construct invariant means on these algebras, and in Section~\ref{sec:cmpts}
we use the results of Section~\ref{sec:embedded} to construct compactifcations.
This has obvious links with quantum \emph{group} compactifications, see
\cite{daws, soltan}, and we explore these links in Section~\ref{sec:bohr}.
In Section~\ref{sec:decomp}, for Kac algebras, we use invariant means to
decompose unitary corepresentations into ``almost periodic'' or ``compact
vector'' parts, and parts which are annihilated by the mean.  In the final
section, we explain carefully how our results relate to the work of
\cite{ss}, and also quickly look at what happens in the cocommutative case
(that is, for the Fourier algebra, instead of for $L^1(G)$).

\subsection{Acknowledgements}

The first named author was supported by the EPSRC grant EP/I026819/1 and
a WCMS-postdoctoral grant.  The 2nd named author was partially
supported by the EPSRC grant EP/I026819/1.  We thank Adam Skalski and
Stuart White for inspiration about the techniques in Section~\ref{sec:means}.
We thank the anonymous referee for very careful proof-reading.

\section{Notation and setting}\label{sec:note}

For a Hilbert space $H$ we write $\mc B(H),\mc B_0(H)$ for the bounded,
respectively, compact, operators on $H$.  We write $\overline{H}$ for the
conjugate Hilbert space: this has set of vectors $\{ \overline{\xi} : \xi\in H \}$
with scalar-multiplication such that $H\rightarrow\overline{H}, \xi\mapsto
\overline{\xi}$ becomes antilinear.  For $x\in\mc B(H)$ we define
$\overline{x}\in\mc B(\overline H)$ by $\overline{x}(\overline{\xi})
= \overline{ x(\xi) }$.  We also define $\top:\mc B(H)\rightarrow
\mc B(\overline{H})$ the ``transpose'' map by
$(\top(x)\overline\xi|\overline\eta) = (\overline{x^*\xi}|\overline\eta)
= (\eta|x^*\xi) = (x\eta|\xi)$ for $\xi,\eta\in H$; we might also write
$x^\top$ for $\top(x)$.

At certain points we use some basic ideas from the theory of Operator Spaces,
especially the various natural tensor products, see \cite[Section~II]{ER}
or \cite[Part~I]{Pisier}.  Similarly, we freely use the concept of the
multiplier algebra of a $C^*$-algebra, see \cite[Appendix~A]{mnw} for an
introduction relevant to this paper.  For $C^*$-algebras we write $\otimes$
for the spacial (minimal) tensor product, and for von Neumann algebras we
write $\vnten$ for the von Neumann algebra tensor product.

Given Banach spaces $E,F$ and a bounded linear map $T:E\rightarrow F^*$,
we can consider the restriction of the Banach space adjoint,
$T^*|_F : F \rightarrow E^*$, and the adjoint of this, $\tilde T = (T^*|_F)^*:
E^{**} \rightarrow F^*$ is the \emph{weak$^*$-extension} of $T$.  If $\kappa_E:
E\rightarrow E^{**}$ is the canonical map, then $\tilde T\kappa_E = T$.
When $E=A$ is a $C^*$-algebra and $F^*=M$ a von Neumann algebra, we speak
of the \emph{normal extension}.

A $C^*$-bialgebra if a pair $(A,\Delta)$ where $A$ is a $C^*$-algebra and
$\Delta:A\rightarrow M(A\otimes A)$ a non-degenerate $*$-homomorphism with
$(\Delta\otimes\id)\Delta = (\id\otimes\Delta)\Delta$.  The examples of
most interest to use are the locally compact quantum groups, \cite{kv},
which, rather roughly, are $C^*$-bialgebras with invariant weights.
For readable introductions see \cite{kus5,vaes}.
We need really rather little of the theory, and so we shall just develop
what we need as we come to it.  We will however setup a little notation.
We think of a quantum group as an abstract object $\G$, and then write
$C_0(\G)$ for the $C^*$-algebra arising.  In general (because of amenability
issues) there might be several natural $C^*$-algebras to consider.
For example, all cocommutative locally compact quantum groups arise as
$C^*_r(G)$, the reduced group $C^*$-algebra of a locally compact group $G$.
However, here we could also consider the full, or universal, group $C^*$-algebra
$C^*(G)$.  Similarly, we can consider $C^u_0(\G)$ the universal $C^*$-algebra
representing $\G$, see \cite{ku}; the main difference is that the weights on
$C^u_0(\G)$ might fail to be faithful.  In this paper, we choose to work with
$C^u_0(\G)$, because it is the maximal choice, but the theory works the same
if we use $C_0(\G)$; compare with Remark~\ref{rem:why_use_full_cstar} below.

The natural morphisms between $C^*$-bialgebras $(A,\Delta_A)$ and
$(B,\Delta_B)$ are non-degenerate $*$-homomorphisms $\theta:B\rightarrow M(A)$
with $\Delta_A\circ\theta = (\theta\otimes\theta)\circ\Delta_B$.  We
``reverse the arrows'' so that if $S,T$ are locally compact semigroups
and $A=C_0(S),B=C_0(T)$ with $\Delta_A,\Delta_B$ the canonical choices, then
we exactly recover the continuous semigroup homomorphisms $S\rightarrow T$.
In the setting of locally compact quantum groups, the paper \cite{mrw}
shows that, so long as we work at the universal level with $C_0^u(\G)$, this
notion of morphism can be recast in a variety of other natural ways.

A unitary corepresentation of a $C^*$-bialgebra $(A,\Delta)$ on a Hilbert space
$H$ is a unitary $U\in M(A\otimes\mc B_0(H))$ with $(\Delta\otimes\id)(U)
= U_{13}U_{23}$.  We sometimes write $H_U$ for $H$ to emphasise the link
between $U$ and $H$.  Given two corepresentations $U,V$ their \emph{tensor product}
is $U \tp V = U_{12} V_{13} \in M(A\otimes\mc B_0(H_U)\otimes\mc B_0(H_V))
= M(A\otimes\mc B_0(H_U \otimes H_V))$.

\section{Semigroups}\label{sec:semigps}

A semitopological semigroup $S$ is a semigroup which is also a (Hausdorff)
topological space, such that the product is separately continuous.  Following
\cite{ss}, but with slightly different terminology, we shall say that a
semitopological semigroup $S$ is (WCH) if $S$ is homeomorphic, as a semigroup,
to a sub-semigroup of contractions on a Hilbert space $H$,
denoted $\mc B(H)_{\|\cdot\|\leq 1}$, with the weak$^*$-topology.
If we replace the weak$^*$-topology by the strong$^*$-topology, we obtain the
property (SCH).

Let $\wch$ be the collection of all WCH semigroups, made into a category with
morphisms the continuous semigroup homomorphisms.  Similarly define $\sch$.
Both of these categories are full subcategories of $\textsf{STS}$, the category
of all semitopological semigroups with continuous semigroup homomorphisms.
Finally, let $\textsf{LCG}$ be the category of locally compact groups with
continuous group homomorphisms, so also $\lcg$ is a full subcategory
of $\textsf{STS}$.  Let $\textsf{CSCH}$, $\textsf{CWCH}$, $\textsf{CSTS}$,
$\textsf{CG}$ be the full subcategories of compact semigroups in, respectively,
$\sch$, $\wch$, $\textsf{STS}$ and $\lcg$.

In this main, first section, we wish to study $\wch$ from a ``noncommutative
topology'' viewpoint: that is, to find a class of $C^*$-algebras with extra
structure such that the commutative algebras exactly correspond to members
of $\wch$, and such that noncommutative algebras still have interesting
properties.

\subsection{The function-free picture}\label{sec:func_free}

In this section, we shall show how to form natural algebras from compact
semigroups of Hilbert space contractions.  This lays the framework, in the
following section, to describe such semigroups purely through certain
generalised C$^*$-bialgebras.

Let $S\in\textsf{CWCH}$ or $\textsf{CSCH}$, say
$S\subseteq\mc B(H)_{\|\cdot\|\leq 1}$, equipped with either the relative
weak$^*$ or SOT$^*$ topologies, as appropriate.  Here we work with compact
semigroups for convenience; see Section~\ref{sec:loc_cmpt_case} below
for some comments on the locally compact case.  Set $A=C(S)$, and let
$SC(S\times S)$ be the space of bounded separately continuous functions
$S\times S\rightarrow\mathbb C$.  Following \cite[Section~2]{runde1} we find
that for $f\in SC(S\times S)$, integrating out one variable gives a $C(S)$
function, and so there is a well-defined map
\[ M(S)\times M(S)\rightarrow\mathbb C; \qquad
(\mu,\lambda)\mapsto \int_S \int_S f(s,t) \ d\mu(s) \ d\lambda(t). \]
This sets up an isometric isomorphism between $SC(S\times S)$ and
$\mc B_{w^*}(M(S),M(S);\mathbb C)$ the space of bounded separately
weak$^*$-continuous bilinear maps $M(S)\times M(S)\rightarrow\mathbb C$,
see \cite[Proposition~2.5]{runde1} (which makes use of
\cite{groth, john1}).

The coproduct for $C(S)$ is naturally realised as a coassociative map
$\Phi:C(S)\rightarrow SC(S\times S); \Phi(f)(s,t) = f(st)$ for $f\in C(S),
s,t\in S$.  Working in the category of Operator
Spaces let $\proten$ be the operator space projective tensor product.  As
$C(S)$ is commutative, $C(S)^*=M(S)$ has the maximal operator space structure,
and so
\[ \mc B_{w^*}(M(S),M(S);\mathbb C) \subseteq (M(S)\proten M(S))^*
= C(S)^{**} \vnten C(S)^{**}. \]
Motivated by this, for an arbitrary $C^*$-algebra $A$ we define
$\mc{CB}_{w^*}(A^*,A^*;\mathbb C)$ to be the collection of
$x\in A^{**}\vnten A^{**}$ such that for each $\mu\in A^*$, the maps
\[ A^*\rightarrow \mathbb C; \quad \lambda\mapsto \ip{x}{\lambda\otimes\mu},
\ip{x}{\mu\otimes\lambda} \]
are both weak$^*$-continuous.  This is equivalent to defining
\[ \mc{CB}_{w^*}(A^*,A^*;\mathbb C) = \{ x\in A^{**}\vnten A^{**}:
(\mu\otimes\id)(x), (\id\otimes\mu)(x)\in A \ (\mu\in A^*) \}. \]
There is a problem here in that for non-commutative $A$, there is no reason why
this need be a subalgebra of $A^{**}\vnten A^{**}$; a counter-example is
given by taking $A=\mc B_0(H)$ for an infinite-dimensional $H$ and letting
$x\in \mc B(H\otimes H)$ be the tensor swap map, compare
Remark~\ref{rem:counter_example1} below.

Nonetheless, we can view $\Phi$ as a $*$-homomorphism $A\rightarrow
A^{**}\vnten A^{**}$ taking values in the subspace $\mc{CB}_{w^*}(A^*,A^*;\mathbb C)$.
Again, working with operator spaces, we have that
\[ A^{**}\vnten \mc B(H) = (A^*\proten\mc B(H)_*)^* = \mc{CB}(A^*,\mc B(H)). \]
As $A=C(S)$ is commutative, $\mc{CB}(A^*,\mc B(H)) = \mc{B}(A^*,\mc B(H))$ and
so we can define $V\in A^{**}\vnten \mc B(H)$ by
\[ \ip{V(\mu)}{\omega_{\xi,\eta}} = \int_S (x\xi|\eta) \ d\mu(x)
\qquad (\mu\in A^*=M(S), \xi,\eta\in H). \]
As $S$ carries the weak$^*$ or SOT$^*$ topology, the map $S\mapsto\mathbb C;
x\mapsto (x\xi|\eta)$ is continuous, and so the integral is defined.
Notice that $\|V\|\leq 1$.

\begin{lemma}
If $S\in\textsf{CSCH}$ then $V\in M(A\otimes\mc B_0(H))$, but this
need not be true for $S\in\textsf{CWCH}$.
\end{lemma}
\begin{proof}
Consider $f\otimes\theta_{\xi,\eta}\in A\otimes\mc B_0(H)$.  Then for
$\mu\in M(S), \omega_{\alpha,\beta}\in\mc B(H)_*$,
\[ \ip{V(f\otimes\theta_{\xi,\eta})}{\mu\otimes\omega_{\alpha,\beta}}
= \ip{V(f\mu)}{\omega_{\xi,\beta}} (\alpha|\eta)
= \int_S (x\xi|\beta)(\alpha|\eta) f(x) \ d\mu(x)
= \int_S (F(x)\alpha|\beta) \ d\mu(x), \]
if $F\in C(S,\mc B_0(H)) = A\otimes\mc B_0(H)$ is defined by
$F(x) = f(x) \theta_{x\xi,\eta}$.  Notice that $F$ is continuous because
$S\rightarrow H, x\mapsto x\xi$ is continuous, $S$ having the SOT$^*$-topology.
Similarly $(f\otimes\theta_{\xi,\eta})V=F'$ where $F'(x)=f(x)
\theta_{\xi,x^*\eta}$ is also continuous.

However, if $S\in\textsf{CWCH}$, for example $S=\mc B(H)_{\|\cdot\|\leq 1}$,
then there are nets $(x_i)$ in $S$ which are weak$^*$-null but not SOT-null,
say $(x_i\xi)$ is not norm null.  With $f=1$, we find that
$F(x_i)=\theta_{x_i\xi,\eta} \not\rightarrow 0=F(0)$ and so $F$ is not
continuous.
\end{proof}

\begin{remark}
We cannot expect $V$ to be unitary.  Indeed, from the calculations above,
for $f\in C(S), \theta\in\mc B_0(H), \omega_{\xi,\eta}\in\mc B(H)_*$ and
$x\in S$, setting $\mu=\delta_x\in M(S)$, we find that
\[ \ip{V(\mu\otimes\omega_{\xi,\eta})}{f\otimes\theta} =
\ip{V}{\mu f \otimes \omega_{\xi,\theta^*\eta}}
= \int_S (f(y) \theta y\xi|\eta) \ d\mu(y)
= (f(x) \theta x\xi|\eta), \]
and this determines $V(\mu\otimes\omega_{\xi,\eta}) \in M(S)\proten\mc B(H)_*$
as $C(S)\otimes\mc B_0(H)$ is weak$^*$-dense in $C(S)^{**}\vnten\mc B(H)$.
If $V(\mu\otimes\omega_{\xi,\eta})$ is approximately $\sum_i
\mu_i\otimes\omega_i$, then
\begin{align*} \ip{V^*V}{\mu\otimes\omega_{\xi,\eta}}
&\approx \sum_i \ip{V^*}{\mu_i\otimes\omega_i}
= \sum_i \overline{ \ip{V}{\mu_i^*\otimes\omega_i^*} }
= \sum_i \overline{ \int_S \ip{y}{\omega_i^*} \ d\mu_i^*(y) } \\
&= \sum_i \int_S \overline{ \ip{y}{\omega_i^*} } \ d\mu_i(y)
= \sum_i \int_S \ip{y^*}{\omega_i} \ d\mu_i(y).
\end{align*}
Now, for all $f,\theta$ we have that
\[ \sum_i \int_S \ip{f(y) \theta}{\omega_i} \ d\mu_i(y)
= \sum_i  \ip{\mu_i}{f} \ip{\theta}{\omega_i}
\approx \ip{V(\mu\otimes\omega_{\xi,\eta})}{f\otimes\theta}
= (f(x)\theta x\xi|\eta), \]
and so by taking limits,
\[ \ip{V^*V}{\mu\otimes\omega_{\xi,\eta}}
= (x^*x\xi|\eta). \]
By contrast, $\ip{1}{\mu\otimes\omega_{\xi,\eta}} = \int_S (\xi|\eta) \ d\mu(y)
= (\xi|\eta)$ and so $V^*V=1$ if and only if $x\in S\implies x^*x=1$, that is,
$S$ is a semigroup of isometries.  Similarly, $VV^*=1$ only when $S$ is a
semigroup of coisometries.
\end{remark}

\begin{lemma}
For any $\omega\in\mc B(H)_*$ we have that $(\id\otimes\omega)(V)\in A=C(S)$
and the collection $\{ (\id\otimes\omega)(V) : \omega\in\mc B(H)_* \}$ separates
the points of $S$.
\end{lemma}
\begin{proof}
By the definition of $V$ we find that
\[ \ip{\mu}{(\id\otimes\omega)(V)} = \int_S \ip{x}{\omega} \ d\mu(x)
\qquad (\mu\in M(S)), \]
and so $(\id\otimes\omega)(V)$ is the continuous function
$S\rightarrow\mathbb C; x\mapsto\ip{x}{\omega}$, a member of $A$.
This collection will clearly separate points as $\omega\in\mc B(H)_*$ varies.
\end{proof}

In general, there is no reason why $\{ (\id\otimes\omega)(V) :
\omega\in\mc B(H)_* \}$ should be a sub-algebra of $C(S)$.  However, by
modifying the Hilbert space, we can obtain this.  Let
\[ F(H) = \bigoplus_{m,n\in\mathbb N_0, m+n\geq 1} H^{m\otimes}
\otimes \overline{H}^{n\otimes}, \]
and define an isometry
$F:\mc B(H)_{\|\cdot\|\leq 1}\rightarrow\mc B(F(H))_{\|\cdot\|\leq 1}$ by
\[ F(x) = \bigoplus_{m,n\in\mathbb N_0, m+n\geq 1} x^{m\otimes}\otimes
\overline{x}^{n\otimes}. \]
A tedious calculation shows that $F$ is weak$^*$-continuous and
SOT$^*$-continuous.  Indeed, we can exploit boundedness and linearity to
reduce the task to showing that $\mc B(H)_{\|\cdot\|\leq 1}\rightarrow
\mc B(H^{m\otimes}\otimes \overline{H}^{n\otimes})_{\|\cdot\|\leq 1}$ is
weak$^*$(SOT$^*$)-continuous, which is clear.  As $S$ is compact,
$F(S)$ is compact and $F$ restricts to a homeomorphism $F:S\rightarrow F(S)$.
Compare this with the discussion in \cite[Section~1.3]{ss} and
\cite[Section~4]{stokke}.

A further complication is that if $0\in S$, then functions of the form
$(\id\otimes\omega)(V)$ can never ``strongly separate'' the points of $S$.
However, let $H'$ be an auxiliary Hilbert space (for example, $H'=\mathbb C$)
and define $S' = \{ x\oplus 1_{H'} : x\in S \} \subseteq\mc B(H\oplus H')$.
Then $S\rightarrow S', x\mapsto x\oplus 1_{H'}$ is a semigroup
homomorphism, and a homeomorphism.  So we may always assume that $0\not\in S$.

\begin{lemma}\label{lem:dense_subalg}
Form $V'$ from $F(S)$ as above.  Then the $*$-algebra generated by
$A_V=\{ (\id\otimes\omega)(V) : \omega\in\mc B(H)_* \}$ is dense in $A_{V'}
=\{ (\id\otimes\omega)(V') : \omega\in\mc B(F(H))_* \}$.
Under the assumption that $0\not\in S$, $A_{V'}$ is a dense $*$-subalgebra
of $C(S)$.
\end{lemma}
\begin{proof}
We find that coefficients of $H^{m\otimes}\otimes \overline{H}^{n\otimes}$
give adjoints of coefficients of $H^{n\otimes}\otimes \overline{H}^{m\otimes}$
and products of coefficients from $H^{m\otimes}\otimes \overline{H}^{n\otimes}$
and $H^{m'\otimes}\otimes \overline{H}^{n'\otimes}$ give coefficients of
$H^{(m+m')\otimes}\otimes \overline{H}^{(n+n')\otimes}$.  Thus $A_{V'}$
is a $*$-subalgebra of $C(S)$.  It's also clear that the $*$-algebra generated
by $A_V$ will be dense in $A_{V'}$.

By Stone-Weierstrass, $A_{V'}$ will be dense in $C(S)$ if $A_{V'}$
``strongly separates points'', \cite[Corollary~8.3, Chapter~V]{conway},
that is, if for each $x\in S$ there is $f\in A_{V'}$ with
$f(x)\not=0$.  But this is clear once we ensure that $0\not\in S$.
\end{proof}

We finally establish some further properties of $(A,\Phi,V)$.
For $\mu,\lambda\in A^*$
define $\ip{\mu\star\lambda}{a} = \ip{\Phi(a)}{\mu\otimes\lambda}$ for $a\in A$.

\begin{lemma}
The bilinear map $\star$ turns $A^*$ into a Banach algebra, and the map
$A^*\rightarrow\mc B(H); \mu\mapsto (\mu\otimes\id)(V)$ becomes a homomorphism
for this product.
\end{lemma}
\begin{proof}
By construction,
\[ \ip{\mu\star\lambda}{a} = \int_S \int_S a(xy) \ d\mu(x) \ d\lambda(y), \]
and so, by using the version of Fubini's Theorem which holds here, compare
\cite[Lemma~2.4]{runde1}, the associativity of the product on $S$ gives that
$\star$ is associative.  Then for $\omega\in\mc B(H)_*$,
\begin{align*} \ip{(\mu\star\lambda\otimes\id)(V)}{\omega}
&= \int_S\int_S (\id\otimes\omega)(V)(xy) \ d\mu(x) \ d\lambda(y)
= \int_S\int_S \ip{xy}{\omega} \ d\mu(x) \ d\lambda(y),
\end{align*}
while
\begin{align*} \ip{(\mu\otimes\id)(V)(\lambda\otimes\id)(V)}{\omega}
&= \int_S \ip{x (\lambda\otimes\id)(V)}{\omega} \ d\mu(x)
= \int_S \int_S \ip{xy}{\omega} \ d\lambda(y) \ d\mu(x),
\end{align*}
so by another application of Fubini, it follows that $(\mu\star\lambda\otimes\id)(V)
= (\mu\otimes\id)(V)(\lambda\otimes\id)(V)$ as required.
\end{proof}

\subsection{The non-commutative framework}

As promised, we can now formulate a certain class of generalised C$^*$-bialgebras
which, in the commutative case, completely capture the classes $\textsf{CSCH}$ and
$\textsf{CWCH}$.  We start with a definition, and then justify that this is
sensible.

\begin{definition}\label{defn:cstareb}
A $C^*$-Eberlein algebra is a quadruple $\mathbb S = (A,\Phi,V,H)$ where:
\begin{enumerate}
\item\label{defn:cstareb:one}
$A$ is a C$^*$-algebra, $\Phi:A\rightarrow A^{**}\vnten A^{**}$
is a unital $*$-homomorphism and $V\in A^{**}\vnten\mc B(H)$ is a contraction;
\item\label{defn:cstareb:two}
$A_V:=\{ (\iota\otimes\omega)(V) : \omega\in\mc B(H)_* \}$ is a subset
of $A\subseteq A^{**}$, and is norm dense in $A$;
\item\label{defn:cstareb:three}
The bilinear product on $A^*$ given by $\ip{\mu\star\nu}{a}
= \ip{\Phi(a)}{\mu\otimes\nu}$ is associative;
\item\label{defn:cstareb:four}
$V$ and $\Phi$ are connected in the sense that the map
$A^*\rightarrow\mc B(H); \mu\mapsto(\mu\otimes\iota)(V)$ is a homomorphism.
\end{enumerate}
\end{definition}

Recall that if $A$ is not unital, then by $\Phi$ being unital we mean,
for example, that the normal extension $\tilde\Phi:A^{**}\rightarrow
A^{**}\vnten A^{**}$ is unital (or equivalently that for any bounded approximate
identity $(e_\alpha)$ of $A$ the net $(\Phi(e_\alpha))$ converges weak$^*$ to
$1\in A^{**}\vnten A^{**}$).

\begin{remark}\label{rem:corep}
Axiom (\ref{defn:cstareb:four}) is exactly the statement that $V$ is a
``corepresentation'', here interpreted to mean that
\[ \Phi\big((\id\otimes\omega)(V)\big)
= (\id\otimes\id\otimes\omega)(V_{13}V_{23})
\qquad (\omega\in\mc B(H)_*). \]
This makes sense, as we can form $V_{13}V_{23}$ in $A^{**}\vnten A^{**}
\vnten \mc B(H)$.  Alternatively, if we let $\tilde\Phi:A^{**} \rightarrow
A^{**}\vnten A^{**}$ be the normal extension of $\Phi$, then this says precisely
that $(\tilde\Phi\otimes\id)(V) = V_{13} V_{23}$.

We also have that $A$ and $V$ uniquely determine $\Phi$.  Indeed, if $\Phi'$
is another coproduct satisfying (\ref{defn:cstareb:three}) and
(\ref{defn:cstareb:four}) then
\[ \Phi'\big((\id\otimes\omega)(V)\big)
= (\id\otimes\id\otimes\omega)(V_{13}V_{23})
= \Phi\big((\id\otimes\omega)(V)\big), \]
for all $\omega\in\mc B(H)_*$.  By (\ref{defn:cstareb:two}) and continuity,
it then follows that $\Phi'=\Phi$
\end{remark}

\subsection{Links with the Haagerup tensor product}
\label{sec:haatenprod}

In Section~\ref{sec:func_free} we carefully introduced to space
$\mc{CB}_{w^*}(A^*,A^*;\mathbb C)$ but in our axioms have made no mention
of it.  By using the Haagerup tensor product, studied in this short section,
we can show that actually automatically $\Phi$ takes values in 
$\mc{CB}_{w^*}(A^*,A^*;\mathbb C)$.

Let $A,B$ be $C^*$-algebras.
The \emph{extended (or weak$^*$) Haagerup tensor product} of $A$ with $B$
is the collection of $x\in A^{**} \vnten B^{**}$ with
\[ x = \sum_{i} x_i \otimes y_i \]
with $\sigma$-weak convergence, for some families $(x_i)$ in $A$ and
$(y_i)$ in $B$, such that $\sum_{i} x_i x_i^*$ and
$\sum_{i} y_i^* y_i$ converge $\sigma$-weakly in $A^{**}$ and $B^{**}$ respectively.
The natural norm is then
\[ \|x\|_{eh} = \inf\Big\{ \Big\|\sum_{i} x_i x_i^*\Big\|^{1/2}
\Big\|\sum_{i} y_i^* y_i\Big\|^{1/2} :
x = \sum_{i} x_i\otimes y_i \Big\}, \]
and we write $A \otimes^{eh} B$ for the resulting normed space.
See \cite{bsm} or \cite{er2} for further details.

Continuing with the notation of the previous section, 
let $(e_i)$ be an orthogonal basis for $H$, so that for $\mu,\mu'\in A^*$
and $\xi,\eta\in H$,
\begin{align*}
= ((\mu\otimes\iota)(V)(\mu'\otimes\iota)(V)\xi|\eta)
&= \sum_i ((\mu\otimes\id)(V)e_i|\eta) ((\mu'\otimes\id)(V)\xi|e_i) \\
&= \sum_i \ip{\mu\otimes\mu'}{(\id\otimes\omega_{e_i,\eta})(V) \otimes
(\id\otimes\omega_{\xi,e_i})(V)}.
\end{align*}
Now, the sum
\[ \sum_i (\id\otimes\omega_{e_i,\eta})(V) \otimes
(\id\otimes\omega_{\xi,e_i})(V) \]
converges in the extended Haagerup tensor product $A\ehten A$.
Furthermore, by conditions (\ref{defn:cstareb:two}), (\ref{defn:cstareb:three})
and (\ref{defn:cstareb:four}), this is equal to
\[ \ip{\Phi((\iota\otimes\omega_{\xi,\eta})(V))}{\mu\otimes\mu'}, \]
and so we conclude that
\[ \Phi((\iota\otimes\omega_{\xi,\eta})(V))
= \sum_i (\id\otimes\omega_{e_i,\eta})(V) \otimes
(\id\otimes\omega_{\xi,e_i})(V). \]
In particular, this shows that $\Phi$ does indeed take values in
$\mc{CB}_{w^*}(A^*,A^*;\mathbb C)$.

\begin{remark}
We would really like to understand
the notion of ``separately continuous'' in the non-commutative setting;
in the previous section we used \cite{runde1} to setup an isomorphism between
$SC(S\times S)$ and $\mc{CB}_{w^*}(A^*,A^*;\mathbb C)$.
Indeed, it would be tempting to believe that maybe we could use the
extended Haagerup tensor product in place of $\mc{CB}_{w^*}(A^*,A^*;
\mathbb C)$.  This appears to be too general a notion however; see
Remark~\ref{rem:haa_no_good} below.
\end{remark}

\subsection{Commutative $C^*$-Eberlein algebras are classical semigroups}

\begin{theorem}\label{thm:comm_unital_case}
For any $S\in\textsf{CWCH}$ or $S\in\textsf{CSCH}$, with $0\not\in S$, we
can canonically construct a $C^*$-Eberlein algebra
$S=(C(S),\Phi_S,V_S,H_S)$.
Conversely, if $\mathbb S=(A,\Phi,V,H)$ with $A$ unital and commutative,
then there is $S\in\textsf{CWCH}$ with $\mathbb S=S$ and $0\not\in S$.
Furthermore, we have that $V_S\in M(A\otimes\mc B_0(H))$ if and only if
$S\in\sch$.
\end{theorem}
\begin{proof}
The first claim follows immediately from the construction in
Section~\ref{sec:func_free}.  Conversely, let $(A,\Phi,V,H)$ be a
$C^*$-Eberlein algebra with $A=C(S)$ for some compact Hausdorff space $S$.
By the previous section, we know that
\[ \Phi:C(S)\rightarrow\mc{B}_{w^*}(M(S),M(S);\mathbb C)=SC(S\times S). \]
So we get a product on $S$ by
\[ f(st) = \ip{\Phi(f)}{\delta_s\otimes\delta_t} \qquad
(s,t\in S,f\in C(S)), \]
and for all $f\in C(S)$ the map $(s,t)\mapsto f(st)$ is separately continuous.
So by condition (\ref{defn:cstareb:three}), $S$ is a semitopological semigroup
and $\Phi=\Phi_S$.

Define $\pi:S\rightarrow\mc B(H)_{\|\cdot\|\leq 1}$ by $\pi(s)
= (\delta_s\otimes\iota)(V)$.  By (\ref{defn:cstareb:four}) we see that
$\pi$ is a homomorphism, and from (\ref{defn:cstareb:two}) it follows that
$\pi$ is injective.  Furthermore, if
$s_\alpha\rightarrow s$ in $S$, then by (\ref{defn:cstareb:two}),
\[ \lim_\alpha \ip{\pi(s_\alpha)}{\omega}
= \lim_\alpha \ip{\delta_{s_\alpha}}{(\id\otimes\omega)(V)}
= \ip{\delta_s}{(\id\otimes\omega)(V)}
= \ip{\pi(s)}{\omega}. \qquad (\omega\in\mc B(H)_*). \]
So $\pi$ is continuous if we give $\mc B(H)_{\|\cdot\|\leq 1}$ the
relative weak$^*$-topology.  As $S$ is compact it
follows that $\pi:S\rightarrow\pi(S)$ is a homeomorphism, if we give $\pi(S)$ the
relative weak$^*$-topology.  We henceforth identify $S$ and $\pi(S)$.
It follows that
\[ \ip{V}{\mu\otimes\omega} = \ip{\mu}{(\id\otimes\omega)(V)}
= \int_S (\id\otimes\omega)(V)(x) \ d\mu(x)
= \int_S \ip{\pi(x)}{\omega} \ d\mu(x), \]
and so $V=V_S$ the canonical corepresentation.
Condition (\ref{defn:cstareb:two}) ensures that $0\not\in S$.

If $S\in\sch$ then $V_S\in M(A\otimes\mc B_0(H))$.  Conversely, if
$V_S\in M(A\otimes\mc B_0(H)) \cong C^b_{SOT^*}(S,\mc B(H))$ then as $V_S$
is the inclusion map $S\rightarrow\mc B(H)_{\|\cdot\|\leq 1}$, it follows that
$(S,\text{weak}^*) \rightarrow (\mc B(H)_{\|\cdot\|\leq 1},SOT^*)$ is
continuous.  So $S$ is SOT$^*$ compact, hence closed, and $S\rightarrow
(\mc B(H)_{\|\cdot\|\leq 1},SOT^*)$ is a homeomorphism onto its range,
as required.
\end{proof}

Let us call $V$ the ``generating corepresentation'' for $\mathbb S$.
We note that $V$ is not unique (the above proof suggests that if $\mathbb S$
is classical, $A=C(S)$ for $S\in\textsf{CWCH}$, then $V$ is always the
``canonical'' $V_S$; but this is only after we have identified $S$ with $\pi(S)$,
where $\pi$ is given by $V$).  Indeed, if $u:H\rightarrow K$ is a unitary,
then $(1\otimes u^*)V(1\otimes u)$ will also generate $(A,\Phi)$.  Slightly
less trivially, if $U\in A^{**}\vnten\mc B(K)$ is such that
$(\id\otimes\omega)(U)\in A$ for all $\omega\in \mc B(K)_*$ and
$\Phi((\id\otimes\omega)(U)) = (\id\otimes\id\otimes\omega)(U_{13}U_{23})$,
then $U\oplus V$ will also be generating.
In particular, we can take direct sums of $V$ with itself.

\begin{remark}
Our philosophy will be that $(A,\Phi)$ is the primary object, and
we assume the existence of \emph{some} $V$.  Axiom (2) just says that $V$
is ``sufficiently large'' in some sense.
\end{remark}

\begin{remark}
We can make use of this to weaken (2).  If all we have is that
$(\id\otimes\omega)(V)\in A$ for all $\omega$, then let $U = V \oplus
V_{12} V_{13} \oplus V_{12}V_{13}V_{14} \oplus \cdots$ on $H\oplus H\otimes H
\oplus H^{\otimes 3} \oplus \cdots$.  As $(\id\otimes\omega\otimes\omega')
(V_{12}V_{13}) = (\id\otimes\omega)(V)(\id\otimes\omega')(V)$, and so forth,
it follows that $A_U$ will be the subalgebra of $A$ generated by $A_V$.
As $\Phi$ is a homomorphism, it's easy to see that $V$ will also be a
corepresentation.  Thus, it is enough that there exists $V$ with
$(\id\otimes\omega)(V)\in A$ for all $\omega$ and such that the closure of
$A_V$ is $*$-closed--- once we have this, we can form $U$ and get $A_U$ to
be a $*$-algebra, which we then require to be dense.
\end{remark}

\subsection{The locally compact case}
\label{sec:loc_cmpt_case}

Let $S$ be a merely \emph{locally} compact, semitopological semigroup.  Then
translates need not leave $C_0(S)$ invariant: for example, if $0\in S$ is a
distinguished point and $S$ has the ``zero multiplication'', $st=0$ for all
$s,t\in S$, then if $S$ is non-compact, for any $t\in S$, $f(ts)=f(st)=f(0)$
for all $s$, so $f(t\cdot) = f(\cdot t) = f(0)1 \not\in C_0(S)$
whenever $f(0)\not=0$.

If we want translations to act on $C_0(S)$ then we need left and right
translations to be \emph{proper} maps (inverse image of compact is compact).
Equivalently, if $(s_\alpha)$ is a net in $S$ tending to infinity (meaning
that for all compact $K\subseteq S$ there is $\alpha_0$ so that
$s_\alpha\not\in K$ for $\alpha\leq\alpha_0$) then also $(ts_\alpha)$
and $(s_\alpha t)$ converge to $\infty$.  Furthermore, once we have this
condition, it is not hard to check that the one-point compactification,
$S_\infty$, becomes semitopological for the product $\infty s = s\infty
= \infty$ for all $s\in S$.

\begin{proposition}
Let $(A,\Phi,V,H)$ be a $C^*$-Eberlein algebra with $A$ commutative and
non-unital.  There is a locally compact, non-compact, $S\in\wch$ with
$A=C_0(S)$ and $\Phi,V,H$ being the canonical choices.  Furthermore,
the one-point compactification $S_\infty$ is canonically a semitopological
semigroup with $S_\infty\in\textsf{CWCH}$.
\end{proposition}
\begin{proof}
We could follow the proof of Theorem~\ref{thm:comm_unital_case}, using
\cite{runde1} to again identify $SC_0(S\times S)$.  However, it is not
now clear that $\pi:S\rightarrow\mc B(H)_{\|\cdot\|\leq 1}$ will be a
homeomorphism, as $S$ is not compact.

However, notice that axiom (\ref{defn:cstareb:two}) tells us that the coefficients
$s\mapsto \ip{\pi(s)}{\omega}$ are in $C_0(S)$ for any $\omega\in\mc B(H)_*$.
So if $s_\alpha\rightarrow\infty$ then $\ip{\pi(ss_\alpha)}{\omega}
= \ip{\pi(s_\alpha)}{\omega\pi(s)} \rightarrow 0$ and similarly 
$\ip{\pi(s_\alpha s)}{\omega} \rightarrow 0$.  As such functions are norm-dense
in $C_0(S)$ this shows that $f(ss_\alpha)\rightarrow 0, f(s_\alpha s)
\rightarrow 0$ for all $s\in S,f\in C_0(S)$.  It follows that translates in
$S$ are proper maps, and so we can turn $S_\infty$ into a semitopological
semigroup in the canonical way.

Furthermore, this argument also shows that $s_\alpha\rightarrow\infty$
in $S$ implies that $\pi(s_\alpha)\rightarrow 0$ weak$^*$ in $\mc B(H)$.
Hence $\pi_\infty:S_\infty\rightarrow\mc B(H)_{\|\cdot\|\leq 1};
S\ni s\mapsto \pi(s), \infty\mapsto 0$ defines an injective, continuous
semigroup homomorphism.  As $S_\infty$ is compact, this is a homeomorphism.
Once again, this gives $0\in\pi_\infty(S_\infty)$, but we can follow the
construction before Lemma~\ref{lem:dense_subalg} and replace $H$ by $H\oplus
\mathbb C$ and replace $0$ by $0\oplus 1_{\mathbb C}$.

Now suppose that $(s_\alpha)$ is a net in $S$ with $\pi(s_\alpha) \rightarrow
\pi(s)$ for some $s\in S$.  Then $\pi_\infty(s_\alpha) \rightarrow
\pi_\infty(s)$ and so $s_\alpha\rightarrow s$ in $S_\infty$, and so
$s_\alpha\rightarrow s$ in $S$.  Thus $\pi$ is a homeomorphism, and so we
can indeed follow the rest of the proof of Theorem~\ref{thm:comm_unital_case}.
\end{proof}

\begin{remark}\label{rem:trivial_compact}
In fact, there is an analogous construction in the non-commutative setting,
which we can view as a way to ``compactify'' in a minimal sense;
see Section~\ref{sec:cmpts} for further details.
Indeed, given $(A,\Phi,V,H)$ with $A$ not unital, consider $B=A\oplus
\mathbb C1$ the unitisation, made into a $C^*$-algebra in the usual way
(e.g.\ consider $B \subseteq M(A)$ to obtain a $C^*$-norm).  So $A$ is a
co-dimension one ideal in $B$, and hence $B^{**} = A^{**} \oplus 1_B$.
One concrete way to see this is to let $A\subseteq\mc B(K)$ be the universal
representation, so $A^{**}\cong A''\subseteq\mc B(K)$.  Then let $A$ act on
$\mc B(K\oplus\mathbb C)$ as $a(\xi\oplus t) = a\xi\oplus 0$ for $a\in A,
\xi\in K,t\in\mathbb C$, and let $1_B$ act as the identity on $K\oplus
\mathbb C$.  Thus we obtain a $*$-representation of $B$, and under
$B^{**}\cong A^{**}\oplus\mathbb C1_B$, we see that $B^{**}$ acts on $K\oplus
\mathbb C$ in exactly the same way.

Define $\Phi_B$ in the obvious way (so $\Phi_B(1)=1\otimes 1$).
Let $H_B=H\oplus\mathbb C$, so we may identify $A^{**}\vnten\mc B(H)$ and
$1_B \vnten \mathbb C$ as subalgebras
of $B^{**}\vnten\mc B(H_B)$.  Under this, let $V_B = V \oplus
1_B\otimes 1_{\mathbb C}$, so $V_B$ is a contraction.
So axioms (\ref{defn:cstareb:one}) and (\ref{defn:cstareb:three}) hold.
Then slices of $V_B$ give us slices of $V$ and scalar multiples of $1_B$,
from which axiom (\ref{defn:cstareb:two}) follows.
As $B^* \cong A^* \oplus \mathbb C$, it
is easy to verify axiom (\ref{defn:cstareb:four}).

We remark that $V_B$ is never unitary (even if $V$ is) and that $V_B$ might
fail to be a multiplier of $B\otimes \mc B_0(H_B)$ even if $V\in
M(A\otimes\mc B_0(H))$ (performing the calculation, you end up needing that
also $V(1\otimes\theta) \in B\otimes\mc B_0(H)$ for all $\theta\in\mc B_0(H)$).

Then, if $A=C_0(S)$, then this abstract construction is exactly the same as
forming $S_\infty$ as in the previous proof, followed by the procedure to
ensure $0\not\in S_\infty$.  (A little care is required in this calculation,
as while $C(S_\infty) \cong C_0(S) \oplus \mathbb C 1$ and hence
$C(S_\infty)^* \cong M(S) \oplus \mathbb C$ this splitting is not naively as
$M(S)\oplus\mathbb C \delta_\infty$.)

Notice that even if $S\in\sch$, it's unlikely that
$S_\infty\in\textsf{CSCH}$; indeed, we'd need that $s_\alpha\rightarrow\infty$
in $S$ implied that $s_\alpha\rightarrow 0$ in the SOT$^*$.
\end{remark}

Thus, in the commutative setting, $C^*$-Eberlein algebras give precisely
\emph{compact} $S\in\wch$ or $S\in\sch$, or give locally compact $S$ such
that $S_\infty$ becomes semitopological.  (Equivalently, those locally
compact semigroups obtained by removing a ``semigroup zero'' from a member of
$\textsf{CWCH}$).

\subsection{Morphisms}\label{sec:semigps:morphisms}

Having now defined the objects we wish to study, we now look at making these
objects into a category; our motivation here is to eventually build a
``compactifcation'' theory.

\begin{definition}
For $i=1,2$ let $\mathbb S_i=(A_i,\Phi_i,V_i,H_i)$ be C$^*$-Eberlein algebras.
A morphism $\mathbb S_1 \rightarrow \mathbb S_2$ is a non-degenerate
$*$-homomorphism $\theta:A_2\rightarrow M(A_1)$ which ``intertwines the
coproducts'', in the weak sense that $\theta^*:A_1^*\rightarrow A_2^*$
is a homomorphism.
\end{definition}

\begin{lemma}\label{lem:mor_bet_cstareb}
With $\mathbb S_1,\mathbb S_2$ as above, and $\theta:A_2\rightarrow M(A_1)$
a non-degenerate $*$-homomorphism, the following are equivalent:
\begin{enumerate}
\item\label{lem:mor_bet_cstareb:one}
$\theta$ is a morphism;
\item\label{lem:mor_bet_cstareb:two}
regard $M(A_1)\subseteq A_1^{**}$ and hence $\theta:A_2\rightarrow
A_1^{**}$ is a homomorphism; let $\tilde\theta:A_2^{**}\rightarrow
A_1^{**}$ be the normal extension.  Similarly let $\tilde\Phi_1:A_1^{**}
\rightarrow A_1^{**}\vnten A_1^{**}$ be the normal extension.
Then we have that $\tilde\Phi_1\circ\theta =
(\tilde\theta\otimes\tilde\theta)\circ\Phi_2$;
\item\label{lem:mor_bet_cstareb:three}
the map $A_1^* \rightarrow \mc B(H_2); \mu\mapsto
(\mu\circ\tilde\theta\otimes\id)(V_2)$ is a homomorphism;
\end{enumerate}
\end{lemma}
\begin{proof}
That (\ref{lem:mor_bet_cstareb:one}) and (\ref{lem:mor_bet_cstareb:two})
are equivalent is routine.  That (\ref{lem:mor_bet_cstareb:one}) and
(\ref{lem:mor_bet_cstareb:three}) are equivalent
follows by using axiom (\ref{defn:cstareb:two}) and Remark~\ref{rem:corep}.
\end{proof}

Suppose that $\theta:A_2\rightarrow A_1$ is a surjection.  Then
$U=(\tilde\theta\otimes\id)(V_2)$ is a generator for $A_1$: indeed, all
that remains to check is that $A_V = \{ (\id\otimes\omega)(V):
\omega\in\mc B(H_2)_* \} = 
\{ \theta(\id\otimes\omega)(V_2): \omega\in\mc B(H_2)_* \}
= \theta(A_{V_2})$ is dense in $A_1$; but this is clear.

\begin{definition}
Let $\G$ be a locally compact quantum group, and let $\mathbb S=
(A,\Phi,V,H)$ be a $C^*$-Eberlein algebra.  A \emph{morphism}
$\G\rightarrow\mathbb S$ is a non-degenerate $*$-homomorphism
$\theta:A\rightarrow M(C_0^u(\G))$ such that:
\begin{enumerate}
\item $\theta^*:C_0^u(\G)^*\rightarrow A^*$ is a homomorphism of Banach
algebras;
\item we require that $U = (\tilde\theta\otimes\id)(V) \in C_0^u(\G)^{**}
\vnten\mc B(H)$ be a ``possibly degenerate unitary corepresentation''.
This means that $(\tilde\Delta_u\otimes\id)(U) = U_{13} U_{23}$, and there is an
orthogonal projection $p\in\mc B(H)$ with $U^*U=UU^*=1\otimes p$.
\end{enumerate}
\end{definition}

Let us motivate this definition from the commutative situation, and then
state some clarifying lemmas.

\begin{proposition}\label{prop:comcasemorphism}
Let $\G=G$ a locally compact group, and $\mathbb S=S$ a classical semigroup
(so $S\in\textsf{CWCH}$ or $\textsf{CSCH}$).  Let $S$ be presented as acting
on the Hilbert space $H$.  There is a bijection between morphisms
$\G\rightarrow\mathbb S$ and representations $\pi:G\rightarrow\mc B(H)$ with a
closed subspace $H_0\subseteq H$, such that, under the decomposition
$H=H_0\oplus H_0^\perp$, $\pi$ acts as unitaries on $H_0$, and as zero on
$H_0^\perp$.  In particular, condition (2) above is automatic in this case.
\end{proposition}
\begin{proof}
Given a morphism $\theta:C(S)\rightarrow M(C_0(G))$, there is a unique continuous
map $\phi:G\rightarrow S$ with $\theta(f)=f\circ\phi$ for each $f\in C(S)$.
Condition (1) implies that $\phi$ is a homomorphism.  Let $\iota:S\rightarrow
\mc B(H)_{\|\cdot\|\leq 1}$ be the inclusion, and let $\pi=\iota\circ\phi$
a group representation of $G$ by contractions.  Thus $p=\pi(e)$ is a
contractive idempotent, hence an orthogonal projection, on $H$; let $H_0=p(H)$.
Then, for any $g\in G$, we see that $\pi(g)=\pi(eg)=p\pi(g) = \pi(ge) = \pi(g)p
= p\pi(g) p$ and so $H_0$ is invariant for $\pi$.  Furthermore, if $\xi\in
H_0^\perp$ then $\pi(g)\xi = \pi(g)p\xi=0$, and so $\pi=0$ on $H_0^\perp$.
As $S\in\wch$, and $\phi$ is continuous, we see that $\pi$ is weakly, hence
also strongly, continuous.

As $V$ ``is'' the map $\iota$, it follows that $U=(\tilde\theta\otimes\id)(V)
\in C_0(G)^{**}\vnten\mc B(H)$ is the canonical corepresentation associated to
$\pi$.  Let us be a little
more careful: for $\omega=\omega_{\xi,\eta}\in\mc B(H)_*$, we find that
\[ (\id\otimes\omega)(U) = \tilde\theta((\id\otimes\omega)(V))
= \tilde\theta\big( (\pi(g)\xi|\eta)_{g\in G} \big) \in C^b(G) = M(C_0(G))
\subseteq C_0(G)^{**}, \]
from which the claim follows.
As $\pi$ is continuous, it follows that really $U$ is a member of
$M(C_0(G)\otimes\mc B(H)) \subseteq C_0(G)^{**}\vnten\mc B(H)$.
It readily follows that $U^*U= UU^* = 1\otimes p$ as claimed.

Clearly we can also reverse this process.
\end{proof}

\begin{lemma}\label{lem:notions_hm_one}
Let $\G,\mathbb S$ be as above, and let $\theta:A\rightarrow M(C_0^u(\G))$ be
a non-degenerate $*$-homomorphism.  Then the following are equivalent:
\begin{enumerate}
\item $\theta^*:C_0^u(\G)^*\rightarrow A^*$ is a homomorphism;
\item if $\tilde\theta:A^{**}\rightarrow C_0^u(\G)^{**}$ is the normal
extension, then $\tilde\Delta_u \circ \tilde\theta|_A = (\tilde\theta\otimes
\tilde\theta)\circ \Phi$.
\item $C_0^u(\G)^* \rightarrow\mc B(H); \mu\mapsto(\mu\circ\tilde\theta
\otimes\id)(V)$ is a homomorphism.
\end{enumerate}
\end{lemma}

\begin{lemma}\label{lem:corepismult}
If $U\in C_0^u(\G)^{**}\vnten\mc B(H)$ with $U^*U=UU^*=1\otimes p$ for a 
projection $p\in\mc B(H)$, we have that $U\in M(C_0^u(\G)\otimes\mc B_0(H))$.
\end{lemma}
\begin{proof}
We follow \cite{ku}.  Let $L^2(\G)$ be the GNS space for the (left) Haar
weight of $\G$, so the reduced algebra $C_0(\G)$ acts canonically on $L^2(\G)$.
Let $\pi:C_0^u(\G)\rightarrow C_0(\G)$ be the ``reducing morphism''.
There is a unitary corepresentation $\mc V$ of $C_0^u(\G)$ on $L^2(\G)$ such
that $(\id\otimes\pi)\Delta_u(x) = \mc V^*(1\otimes\pi(x))\mc V$ for all
$x\in C_0^u(\G)$; see \cite[Section~6]{ku}.

With $\tilde\Delta_u:C_0^u(\G)^{**} \rightarrow C_0^u(\G)^{**} \vnten C_0^u(\G)^{**}$
and $\tilde\pi:C_0^u(\G)^{**} \rightarrow C_0^u(\G)$ being the normal extensions,
arguing by weak$^*$-continuity, we also have that
\[ (\id\otimes\tilde\pi)\tilde\Delta_u(x) = \mc V^*(1\otimes\tilde\pi(x))\mc V
\qquad (x\in C_0^u(\G)^{**}). \]
It then follows that
\[ \mc V^*_{12} (\tilde\pi\otimes\id)(U)_{23} \mc V_{12}
= (\id\otimes\tilde\pi\otimes\id)(U_{13}U_{23})
= U_{13} (\tilde\pi\otimes\id)(U)_{23}. \]
As $\mc V \in M(C_0^u(\G)\otimes\mc B_0(L^2(\G)))$, it follows that
\begin{align*} (U(1\otimes p))_{13}
&= U_{13} (\tilde\pi\otimes\id)(UU^*)_{23}
= \mc V^*_{12} (\tilde\pi\otimes\id)(U)_{23} \mc V_{12}
(\tilde\pi\otimes\id)(U)_{23}^* \\
&\in M\big( C_0^u(\G) \otimes \mc B_0(L^2(\G)) \otimes \mc B_0(L^2(\G))
\big). \end{align*}
Thus $U(1\otimes p) \in M(C_0^u(\G)\otimes\mc B_0(L^2(G)))$.
If we represent $C_0^u(\G)$ faithfully and non-degenerately on an auxiliary
Hilbert space $K$, then we can regard $U$ as a partial isometry with initial
and final spaces $K\otimes P(H)$, this is enough to show that also
$U\in M(C_0^u(\G)\otimes \mc B_0(L^2(\G)))$.
\end{proof}

\begin{remark}
In the proof of Proposition~\ref{prop:comcasemorphism} we showed that in the
commutative situation, a contractive corepresentation is necessarily a
(possibly degenerate) unitary corepresentation.  Unfortunately, we don't
know this to be true for locally compact quantum groups, even in
the cocommutative case, see \cite{bs} and \cite{bds}.
\end{remark}

\begin{remark}\label{rem:lcqg_is_cstareb}
Let $\mc V\in M(C_0^u(\G)\otimes\mc B_0(L^2(\G)))$ again denote the universal
left regular representation, see \cite[Section~6]{ku}.
Then $(C_0^u(\G),\Delta_u,\mc V,L^2(\G))$
satisfies all the axioms for a $C^*$-Eberlein algebra; for axiom
(\ref{defn:cstareb:two}) see the discussion after \cite[Proposition~5.1]{ku}.
Hence the category of locally compact quantum groups
forms a full sub-category of the $C^*$-Eberlein algebras.
Unfortunately, due to the above comment, we have to restrict the natural
morphisms between locally compact quantum groups and $C^*$-Eberlein algebras.
\end{remark}

\subsection{Banach algebra issues}\label{sec:baissues}

Let $(A,\Phi,V,H)$ be a $C^*$-Eberlein algebra, so that $A^*$ is a Banach
algebra.  We claim that it is a \emph{dual Banach algebra}, meaning that the
multiplication is separately weak$^*$-continuous.  Alternatively, we can
always turn $A^{**}$ into an $A^*$ bimodule via
\[ \ip{x\cdot\mu}{\lambda} = \ip{x}{\mu\star\lambda}, \quad
\ip{\mu\cdot x}{\lambda} = \ip{x}{\lambda\star\mu}
\qquad (x\in A^{**}, \mu,\lambda\in A^*). \]
Then $A^*$ is a dual Banach algebra if and only if $A$ is a submodule of
$A^{**}$.

Let us introduce some notation: for $\mu\in A^*$ let
$V_\mu = (\mu\otimes\id)(V) \in \mc B(H)$.  By axiom (\ref{defn:cstareb:four}),
we know that $\mu\mapsto V_\mu$ is a representation of $A^*$.
By axiom (\ref{defn:cstareb:two}), this map
$\mu\mapsto V_\mu$ is both injective, and weak$^*$-weak$^*$-continuous.

\begin{proposition}\label{prop:get_dba}
$A^*$ is a dual Banach algebra.
\end{proposition}
\begin{proof}
We verify that $A$ is a submodule of $A^{**}$.  By axiom (\ref{defn:cstareb:two}),
and continuity, it's enough to check that $A_V$ is a submodule.  However, for
$a=(\id\otimes\omega)(V)\in A$ and $\mu,\lambda\in A^*$,
\[ \ip{a\cdot\mu}{\lambda} = \ip{V}{\mu\star\lambda\otimes\omega}
= \ip{V_\mu V_\lambda}{\omega} = \ip{V}{\lambda\otimes \omega V_\mu}, \]
so $a\cdot\mu = (\id\otimes \omega V_\mu)(V) \in A_V$.  Similarly, one can
check that $\mu\cdot a = (\id\otimes V_\mu\omega)(V) \in A_V$.

Alternatively, we note that $\mu\mapsto V_\mu$, when restricted to the closed
unit ball of $A^*$, must be a homeomorphism onto its range (as the ball is
weak$^*$-compact).  Then $\{ V_\mu : \|\mu\|\leq 1 \}$ is a weak$^*$-compact
sub-semigroup of $\mc B(H)_{\|\cdot\|\leq 1}$ and is hence semitopological.
Thus $\{ \mu\in A^* : \|\mu\|\leq 1\}$ with the product $\star$ is a
semitopological semigroup, and the result follows.
\end{proof}

Recall that if $A$ is a C$^*$-algebra and $M(A)$ its multiplier algebra, then
any $\mu\in A^*$ is strictly continuous, and so has an extension to $M(A)$.
Alternatively, we can regard $M(A)$ as a (norm) closed subalgebra of $A^{**}$,
and then $A^*$ acts by restriction on $M(A)$.  These two maps $A^*\rightarrow
M(A)^*$ are the same.
If $\theta:A_2 \rightarrow M(A_1)$ is a morphism $\mathbb S_1\rightarrow
\mathbb S_2$, we get an induced homomorphism $\theta^*:A_1^*\rightarrow A_2^*$;
but this is a slight abuse of notation, as $\theta$ might not map into $A_1$,
and so $\theta^*$ might not be weak$^*$-continuous.

We now study some permanence properties of $C^*$-Eberlein algebras, making
use of the above observations.

\begin{proposition}\label{prop:when_can_quotient}
Let $(A,\Phi,V,H)$ be a $C^*$-Eberlein algebra, let $B$ be a $C^*$-algebra,
and let $\theta:A\rightarrow B$ be a surjection.  Set $I$ to be the kernel of
$\theta$, so $B\cong A/I$.  Then there exists
$\Phi':B\rightarrow B^{**}\vnten B^{**}$ and $V'\in B^{**}\vnten\mc B(H')$
making $(B,\Phi',V',H')$ into a $C^*$-Eberlein algebra with $\theta$ a
morphism, if, and only if, $I^\perp$ is a subalgebra of $A^*$.
\end{proposition}
\begin{proof}
Note that $\theta^*:B^*\rightarrow A^*$ is an isometry onto its range, which
is $I^\perp = \{ \mu\in A^* : \ip{\mu}{a}=0\ (a\in I) \}$.  Then 
$\theta:A^{**}\rightarrow B^{**}$ is a surjection with kernel $I^{\perp\perp}
\cong I^{**}$.  As $I^{\perp\perp}$ is an ideal in $A^{**}$, there is a central
projection $p\in A^{**}$ with $(1-p)A^{**} = I^{\perp\perp}$ and so $B^{**}
\cong A^{**}/I^{\perp\perp}$ is isomorphic to $pA^{**}$.  Notice also that
the map $\mu\mapsto p\mu$ defines a surjection $A^*\rightarrow I^\perp$.

If $\theta$ is a morphism, then $I^\perp\cong B^*$ is a subalgebra of $A^*$.
Conversely, if $I^\perp$ is a subalgebra of $A^*$, then consider
$\Phi'':A\rightarrow B^{**}\vnten B^{**}, a\mapsto (p\otimes p)\Phi(a)$ which
is a $*$-homomorphism as $p$ is central.  Now let $\mu,\lambda\in A^*$ so
that $p\mu,p\lambda\in I^\perp$ and so by assumption $p\mu \star p\lambda
\in I^\perp$.  Thus, for $a\in I$ we have that $0 = \ip{p\mu\star p\lambda}{a}
= \ip{(p\otimes p)\Phi(a)}{\mu\otimes\lambda}$.  This shows that $\Phi''(I)
= \{0\}$ and so there is a well-defined $*$-homomorphism
$\Phi':B=A/I\rightarrow B^{**} \vnten B^{**}$ given by
$\Phi'(a+I) = \Phi''(a) = (p\otimes p)\Phi(a)$.

So $\Phi'$ induces a product on $B^*$.
Then, for $a\in A$ and $\mu,\lambda\in B^*\cong I^\perp$, we see that
\begin{align*} \ip{\theta^*(\mu\star\lambda)}{a}
&= \ip{\Phi'(\theta(a))}{\mu\otimes\lambda}
= \ip{(p\otimes p)\Phi(a)}{\mu\otimes\lambda}
= \ip{\Phi(a)}{\mu\otimes\lambda} = \ip{\mu\star\lambda}{a} \\
&= \ip{\theta^*(\mu)\star \theta^*(\lambda)}{a}, \end{align*}
as $\theta^*:B^*\cong I^\perp\rightarrow A^*$ is just the inclusion.

Finally, we set $H'=H$ and $V'=(p\otimes 1)V
\in B^{**}\vnten\mc B(H')$, so that for $\omega\in\mc B(H')_*$,
\[ (\id\otimes\omega)(V') = p(\id\otimes\omega)(V)
= \theta^{**}\big( (\id\otimes\omega)(V) \big)
= \theta( (\id\otimes\omega)(V) ) \in B, \]
as $(\id\otimes\omega)(V)\in A$.  As $\theta$ is onto, it follows that
$\{ (\id\otimes\omega)(V') : \omega\in\mc B(H')_* \}$ is dense in $B$.
Also $\Phi'((\id\otimes\omega)(V')) =
(p\otimes p)\Phi((\id\otimes\omega)(V)) =
(\id\otimes\id\otimes\omega)(V'_{13} V'_{23})$ and so $V'$ is a
corepresentation, as required.
\end{proof}

We remark that the ``coproduct'' $\Phi'$ on $B$ is uniquely determined by
the product on $B^*$.  As $\theta:A\rightarrow B$ is onto, the coproduct
$\Phi'$ is hence unique.

\begin{proposition}\label{prop:factoring}
Let $\mathbb S_1\rightarrow\mathbb S_2$ be a morphism induced by
$\theta:A_2\rightarrow M(A_1)$.  With $I=\ker\theta$ we have that $I^\perp$
is a subalgebra of $A_2^*$ and so $B=\theta(A_2)\subseteq M(A_1)$ has the
structure of a $C^*$-Eberlein algebra.  The inclusion $B\rightarrow
M(A_1)$ is a morphism.
\end{proposition}
\begin{proof}
That we have to work in $M(A_1)$ complicates matters.  Let $\mu\in I^\perp
\cong B^*$ so by Hahn-Banach there is $\mu_0\in M(A_1)^*$ extending $\mu$.
Viewing $M(A_1)$ as a subalgebra of $A_1^{**}$ there is $\mu_1\in A_1^{***}$
extending $\mu_0$ and so there is a bounded net $(\mu_i)$ in $A_1^*$ converging
weak$^*$ to $\mu_2$.  That is,
\[ \lim_i \ip{\mu_i}{\theta(a)} = \ip{\mu}{a} \qquad (a\in A_2). \]
Similarly choose $(\lambda_i)$ for $\lambda\in I^\perp$.  Then, for $a\in I$,
\begin{align*} \ip{\mu\star\lambda}{a} &=
\ip{\mu}{\lambda\cdot a} = \lim_i \ip{\mu_i}{\theta(\lambda\cdot a)}
= \lim_i \ip{\theta^*(\mu_i)\star\lambda}{a} \\
&= \lim_i \ip{\lambda}{a\cdot \theta^*(\mu_i)}
= \lim_i \lim_j \ip{\lambda_j}{\theta(a\cdot \theta^*(\mu_i))}
= \lim_{i,j} \ip{\theta^*(\mu_i)\star\theta^*(\lambda_j)}{a} \\
&= \lim_{i,j} \ip{\theta^*(\mu_i\star\lambda_j)}{a} = 0, \end{align*}
as $\theta(a)=0$.  Thus $\mu\star\lambda\in I^\perp$ as claimed.

The inclusion $i:B\rightarrow M(A_1)$ is a $*$-homomorphism, and it is
non-degenerate as $\lin\{ ba : b\in A, a\in A_1 \} = \lin\{ \theta(a')a:
a'\in A_2,a\in A_1 \}$ is dense in $A_1$ as $\theta$ is non-degenerate.
Then $i$ is a morphism if and only if $i^*:A_1^*\rightarrow B^*$ is a
homomorphism.  However, $B^*$ is isomorphic (by construction, as a Banach
algebra) to $I^\perp\subseteq A_2^*$ and $i^*$, when composed with this
inclusion $B^*\rightarrow A_2^*$, is the map $\theta^*$ which is a
homomorphism.  So $i^*$ is a homomorphism as well, and $i$ is a morphism.
\end{proof}

\subsection{Embedded $C^*$-Eberlein algebras}\label{sec:embedded}

We now start a process of thinking about compactifcations, which as we will
see in Section~\ref{sec:cmpts} leads naturally to considering a locally compact
quantum group $\G$ and C$^*$-Eberlein algebras $(A,\Phi,V,H)$ with $A$ a
subalgebra of $M(C_0^u(\G))$ and $\Phi$ being ``induced'' by the coproduct of
$C_0^u(\G)$.  The purpose of this section is to make this concept precise, and
to characterise when it can occur.

We start with a useful technical proposition.

\begin{proposition}\label{prop:subalg}
Let $A$ be a $C^*$-algebra, $H$ a Hilbert space, and $B\subseteq M(A)$
a sub-$C^*$-algebra.  Let \[ X_0=\{ T\in M(A\otimes\mc B_0(H)) :
(\id\otimes\omega)(T) \in B \ (\omega\in\mc B(H)_*) \}. \]
Then $X_0$ is a closed, $*$-closed subspace of $M(A\otimes\mc B_0(H))$.
For each $T\in X_0$ there is $S\in B^{**}\vnten \mc B(H)$ with
$(\id\otimes\omega)(T) = (\id\otimes\omega)(S) \in B \subseteq B^{**}$ for all
$\omega\in\mc B(H)_*$.  If $T_1,T_2\in X_0$ are such that $T_1T_2 \in X_0$,
and $T_i\mapsto S_i$, then $T_1T_2 \mapsto S_1S_2$.
\end{proposition}
\begin{proof}
For $T\in X_0$, the map $\mc B(H)_* \rightarrow M(A); \omega \mapsto
(\id\otimes\omega)(T)$ is completely bounded, and hence so is the corestriction
$\mc B(H)_*\rightarrow B$, and hence so is the inclusion
$B\subseteq B^{**}$, which finally induces $S\in B^{**}\vnten\mc B(H)$.
We have that $\|S\|=\|T\|$.
We first claim that this map is a $*$-map.  Notice that
$(\id\otimes\omega)(T^*) = (\id\otimes\omega^*)(T)^* \in B$ so $T^*\in X_0$.
Then, for $\mu\in B^*$,
\[ \ip{\mu}{(\id\otimes\omega)(S^*)} = \overline{ \ip{S}{\mu^*\otimes\omega^*}}
= \overline{ \ip{\mu^*}{(\id\otimes\omega^*)(T)} }
= \ip{\mu}{(\id\otimes\omega)(T^*)}, \]
so $T^*\mapsto S^*$ as claimed.

Now let $T_1,T_2\in X_0$ map to $S_1,S_2$.  Let $(e_i)$ be an orthonormal basis
of $H$.  For $\xi,\eta\in H$,
\[ (\id\otimes\omega_{\xi,\eta})(T_1T_2) = \sum_i (\id\otimes\omega_{\xi,\eta})
(T_1 (1\otimes\theta_{e_i,e_i}) T_2) =
\sum_i (\id\otimes\omega_{e_i,\eta})(T_1) (\id\otimes\omega_{\xi,e_i})(T_2). \]
This sum converges in the strict topology on $M(A)$; compare with the discussion
in \cite[Section~5.5]{kv}, or \cite[Lemma~A.5]{kv}, for example.
However, each summand is a member of
$B\subseteq B^{**}$ and one can similarly prove that the sum converges
in the $\sigma$-weak topology on $B^{**}$, compare with
Section~\ref{sec:haatenprod} above.  Note in passing that we don't see
why the sum need converge in norm in $B$, and so there is no reason why $X=X_0$.
Continuing, we see that in $B^{**}$, the sum is also equal to
\[ \sum_i (\id\otimes\omega_{e_i,\eta})(S_1) (\id\otimes\omega_{\xi,e_i})(S_2)
= (\id\otimes\omega_{\xi,\eta})(S_1S_2). \]
By assumption, $T_1T_2\in X_0$ and so $(\id\otimes\omega)(T_1T_2)\in B$.
The argument now shows this to be equal to $(\id\otimes\omega)(S_1S_2)$
as required.
\end{proof}

\begin{remark}\label{rem:counter_example1}
We shall apply this proposition in the case when $B$ is unital.  Then, for
example, if $U\in X_0$ is unitary, then $U^*U=UU^*=1_A\otimes I_H=1_B\otimes I_H$
and so the 2nd part of the proposition allows us to conclude that the image
of $U$ in $B^{**}\vnten\mc B(H)$ is also unitary.

In general, the sum may indeed fail to converge in norm.  For example,
if $A=B=\mc B_0(H)$ then $M(A\otimes\mc B_0(H)) = M(B_0(H)\otimes\mc B_0(H))
\cong \mc B(H\otimes H)$.  Let $\Sigma\in \mc B(H\otimes H)$ be the tensor
swap map, so for $\xi,\eta,\alpha,\beta\in H$,
\[ ((\id\otimes\omega_{\xi,\eta})(\Sigma)\alpha|\beta)
= (\Sigma(\alpha\otimes\xi)|\beta\otimes\eta)
= (\xi\otimes\alpha|\beta\otimes\eta)
= (\xi|\beta) (\alpha|\eta) = (\theta_{\xi,\eta}\alpha|\beta), \]
and so $(\id\otimes\omega_{\xi,\eta})(\Sigma) = \theta_{\xi,\eta} \in B$.
So $\Sigma\in X_0$.  However, $\Sigma^2 = I_H\otimes I_H$ and so
$\Sigma^2\not\in X_0$.
\end{remark}

Suppose we start with a locally compact quantum group $\G$ and a $C^*$-Eberlein
algebra $\mathbb S=(A,\Phi,V,H)$ together with a morphism $\G\rightarrow\mathbb S$
given by $\theta:A\rightarrow M(C_0^u(\G))$.  If $B = \theta(A) \subseteq M(C_0^u(\G))$
then by Proposition~\ref{prop:factoring} we can give $B$ a unique structure
as a $C^*$-Eberlein algebra, in such a way that the inclusion
$\phi:B\rightarrow M(C_0^u(\G))$ is a morphism.  Indeed, this uses
Proposition~\ref{prop:when_can_quotient} to form $(B,\Phi_B,V_B,H)$ with
$V_B=(\theta_0^{**}\otimes\id)(V) \in B^{**}\vnten\mc B(H)$, where 
$\theta_0:A\rightarrow B$ is the corestriction of $\theta$.  Let us pin down
the ``corepresentation'' $V_B$ a little further.

\begin{proposition}
Continuing with the notation of the previous paragraph, there is a subspace $K$
of $H$ which is invariant for $V_B$ and such that if $U'$ is the restriction of
$V_B$ to $K$ then $U'$ is unitary and $(B,\Phi_B,U',K)$ is a $C^*$-Eberlein
algebra.  Furthermore, there is a unitary corepresentation $U$ of $C_0^u(\G)$
on $K$ such that
\[ \big\{ (\id\otimes\omega)(U) : \omega\in\mc B(K)_* \big\} \]
is dense in $B$, and $U'$ is generated from $U$ by the procedure of
Proposition~\ref{prop:subalg}.
\end{proposition}
\begin{proof}
By the definition of $\theta$ being a morphism,
$U_0 = (\tilde\theta\otimes\id)(V) \in C_0^u(\G)^{**}\vnten\mc B(H)$
is such that $U_0^*U_0=U_0U_0^*=1\otimes p$ for some projection $p$ on $H$.
By Lemma~\ref{lem:corepismult} we know that actually
$U_0\in M(C_0^u(\G)\otimes\mc B(H))$, so we may let $U$ be the restriction of $U_0$
to a unitary corepresentation of $C_0^u(\G)$ in
$M(C_0^u(\G)\otimes\mc B_0(K))$, where $K=p(H)$.

Now let $\tilde\phi:B^{**}\rightarrow C_0^u(\G)^{**}$ be the normal extension
of $\phi$ and notice that $\tilde\phi\circ\theta_0^{**} = \tilde\theta$.  Hence
$(\tilde\phi\otimes\id)(V_B) = U_0$.  If $i:K\rightarrow H$ is the inclusion,
then
\[ (\id\otimes\omega)(U) = (\id\otimes i\omega p)(U_0)
= \tilde\theta( (\id\otimes i\omega p)(V) )
= \theta( (\id\otimes i\omega p)(V) )
\qquad ( \omega\in\mc B(K)_* ) \]
and this is a member of $B$ by axiom (\ref{defn:cstareb:two}).
Furthermore, $B$ is in fact the closure of the set
$\{ (\id\otimes\omega)(U) : \omega\in\mc B(K)_* \}$.
So can apply the procedure of Proposition~\ref{prop:subalg} to $U$ to yield
$U'\in B^{**}\vnten\mc B(K)$.  By construction,
$(\id\otimes\omega)(U') = (\id\otimes\omega)(U) = \tilde\phi((\id\otimes
i\omega p)(V_B)) = \phi((\id\otimes i\omega p)(V_B))$ as 
$(\id\otimes i\omega p)(V_B)\in B$, for any $\omega\in B(K)_*$.  It follows
that $U'$ is simply the compression of $V_B$ to $K$, that is,
$U' = (1\otimes p)V_B(1\otimes i)$.

However, by Proposition~\ref{prop:subalg}, as $U^*U=UU^*=I$, also $U'$ is
unitary.  So $K$ is an invariant subspace of $V_B$, and hence $(B,\Phi_B,U',K)$
is a $C^*$-Eberlein algebra.
\end{proof}

Let us think when we can reverse this process: that is, if $U\in M(C_0^u(\G)
\otimes\mc B(K))$ is a unitary corepresentation of $C_0^u(\G)$, when is
there a $C^*$-Eberlein algebra $(A,\Phi,V,H)$ and a morphism
$\theta:A\rightarrow M(C_0^u(\G))$ such that $U$ arises from the construction
above?  It is clearly necessary that $\{ (\id\otimes\omega)(U) :
\omega\in\mc B(K)_* \}$ is dense in a $C^*$-algebra $B\subseteq M(C_0^u(\G))$
(which will be the image of $\theta$).  This turns out to also be sufficient.

\begin{theorem}\label{thm:emb_cseb}
Let $U\in M(C_0^u(\G)\otimes\mc B(K))$ be a unitary corepresentation of
$C_0^u(\G)$ such that the closure of $\{ (\id\otimes\omega)(U) :
\omega\in\mc B(K)_* \}$ is a $C^*$-algebra $B\subseteq M(C_0^u(\G))$.
Let $U'\in B^{**}\vnten\mc B(K)$ be given by Proposition~\ref{prop:subalg}
applied to $U$.  Then there is a unique $\Phi_B$ such that
$\mathbb S=(B,\Phi_B,U',K)$ is a C$^*$-Eberlein algebra with the inclusion
$B\rightarrow M(C_0^u(\G))$ inducing a morphism $\G\rightarrow\mathbb S$.
\end{theorem}

Before giving the proof, we need to introduce a little Banach algebraic
background.  For a Banach algebra $\mf A$ turn $\mf A^*$ into an
$\mf A$-bimodule in the usual way.  For a functional $\mu\in\mf A^*$,
define the orbit maps
\[ L_\mu:\mf A\rightarrow\mf A^*, a\mapsto \mu\cdot a, \qquad
R_\mu:\mf A\rightarrow\mf A^*, a\mapsto a\cdot\mu. \]
Let $\kappa:\mf A\rightarrow\mf A^{**}$ be the canonical map from a Banach
space to its bidual.  Then it's easy to see that $L_\mu^*\circ\kappa = R_\mu$
and $R_\mu^* \circ \kappa = L_\mu$.  From Gantmacher's theorem, it follows
that $R_\mu$ is a weakly-compact map if and only if $L_\mu$ is, and in this
case we say that $\mu$ is \emph{weakly almost periodic}, written $\mu\in
\wap(\mf A)$.  Then $\wap(\mf A)$ is a closed (two-sided) submodule of $\mf A^*$.

Now recall the Arens products on $\mf A^{**}$: these are the two natural algebra
products on $\mf A^{**}$ obtained by extending the product from $\mf A$ by
weak$^*$-continuity, either on the left or the right.  More algebraically,
for $\Phi,\Psi\in\mf A^{**}$ and $\mu\in\mf A, a\in\mf A$ we successively define
\begin{gather*} \ip{\Phi\aone\Psi}{\mu} = \ip{\Phi}{\Psi\cdot\mu},
\qquad \ip{\Phi\atwo\Psi}{\mu} = \ip{\Psi}{\mu \cdot \Phi} \\
\ip{\Psi\cdot\mu}{a} = \ip{\Psi}{\mu\cdot a}, \qquad
\ip{\mu\cdot\Phi}{a} = \ip{\Phi}{a\cdot\mu}. \end{gather*}
Then $\aone,\atwo$ are Banach algebra products on $\mf A^{**}$ making
$\kappa$ a homomorphism.  We say that a closed submodule $X$ of $\mf A^*$
is \emph{introverted} if $\Phi\cdot\mu,\mu\cdot\Phi\in X$ for all $\mu\in X,
\Phi\in\mf A^{**}$.  In this case, it is easy to see that $\aone$ and
$\atwo$ drop to well-defined products on $X^* = \mf A^{**} / X^\perp$.
One can then show that the resulting products on $X^*$ are the same if
and only if $X\subseteq\wap(\mf A)$ (and we have that $\wap(\mf A)$ is
always introverted).  Furthermore, in this case, the product in $X^*$ is
separately weak$^*$-continuous, and so $X^*$ becomes a dual Banach
algebra.  For a self-contained proof of this, see
\cite[Proposition~2.4]{daws1}; compare \cite[Lemma~1.4]{lr}.

\begin{proof}[Proof of Theorem~\ref{thm:emb_cseb}]
Firstly, $\theta:B\rightarrow M(C_0^u(\G))$ will be non-degenerate.  Indeed,
as $U$ is a unitary multiplier, $\{ U(a\otimes\theta) : a\in C_0^u(\G),
\theta\in\mc B_0(K) \}$ will have dense linear span in $C_0^u(\G)\otimes
\mc B_0(K)$ and hence
\begin{align*} \overline{\lin}\{ (\id\otimes\omega)(U)a  &
  : a\in C_0^u(\G), \omega\in\mc B(K)_* \} \\
&= \overline{\lin}\{ (\id\otimes\theta\omega)(U(a\otimes 1))
  : a\in C_0^u(\G), \theta\in\mc B_0(K), \omega\in\mc B(K)_* \} \\
&= \overline{\lin}\{ (\id\otimes\omega)(U(a\otimes\theta))
  : a\in C_0^u(\G), \theta\in\mc B_0(K), \omega\in\mc B(K)_* \} \\
&= \overline{\lin}\{ (\id\otimes\omega)(a\otimes\theta)
  : a\in C_0^u(\G), \theta\in\mc B_0(K), \omega\in\mc B(K)_* \}
= C_0^u(\G),
\end{align*}
as required to show that $\theta(B)C_0^u(\G)$ has dense span in $C_0^u(\G)$.

We have that $C_0^u(\G)^*$ is a Banach algebra, and so $C_0^u(\G)^{**}$ is
a $C_0^u(\G)^*$-bimodule.  If we regard $B$ as a subalgebra of $M(C_0^u(\G))$,
which in turn we regard as a subalgebra of $C_0^u(\G)^{**}$, then we claim that
$B$ is a $C_0^u(\G)^*$-submodule.  Indeed, for $b=(\id\otimes\omega)(U)\in B$
and $\mu\in C_0^u(\G)^*$, we find that for all $\lambda\in C_0^u(\G)^*$,
\[ \ip{\mu\cdot b}{\lambda} = \ip{\lambda\star\mu\otimes\omega}{U}
= \ip{\lambda\otimes\mu\otimes\omega}{U_{13}U_{23}}
= \ip{\lambda\otimes\omega}{U(1\otimes(U_\mu))}
= \ip{(\id\otimes U_\mu\omega)(U)}{\lambda}, \]
so $\mu\cdot b = (\id\otimes U_\mu\omega)(U)\in B$.
Similarly $b\cdot\mu = (\id\otimes \omega U_\mu)(U)\in B$.
We next claim that $B\subseteq\wap(C_0^u(\G)^*)$.  Indeed, for $b$ of the form
above, the map $C_0^u(\G)^*\rightarrow B; \mu\mapsto \mu\cdot b =
(\id\otimes\omega U_\mu)(U)$ factors through the map $\mc B(H)\rightarrow
\mc B(H)_*; T\mapsto \omega T$ which is weakly-compact.  As $\wap(C_0^u(\G)^*)$
is closed, the claim follows.

From the discussion above, the Arens products on $C_0^u(\G)^{***}$ drop to
a well-defined product on $B^* = C_0^u(\G)^{***} / B^\perp$ turning $B^*$ into
a dual Banach algebra.  As $C_0^u(\G)^*$ is a
completely contractive Banach algebra (it is the predual of the
Hopf-von Neumann algebra $(C_0^u(\G)^{**},\tilde\Delta_u)$)
we may work in the category of Operator Spaces to conclude that $B^*$ is also
a completely contractive dual Banach algebra.  That is, the product on $B^*$
induces a complete contraction $\Phi_*:B^*\proten B^*\rightarrow B^*$ and the
adjoint gives a complete contraction $B^{**}\rightarrow B^{**}\vnten B^{**}$;
let $\Phi_B$ be this map restricted to $B$.

We can describe $\Phi_B$ more concretely as follows.  Let $\mu\in B^*$ and let
$\mu_0\in C_0^u(\G)^{***}$ be a Hahn-Banach extension.  There is hence a
bounded net $(\mu_i)$ in $C_0^u(\G)^*$ converging weak$^*$ to $\mu_0$.  For
$\lambda\in B^*$ similarly find $(\lambda_j)$.  Then, for $b\in B$,
\[ \ip{\Phi_B(b)}{\mu\otimes\lambda} = \ip{\mu\star\lambda}{b}
= \ip{\mu_0}{\lambda_0\cdot b} = \lim_i \ip{\lambda_0\cdot b}{\mu_i}
= \lim_i \ip{\lambda_0}{b\cdot\mu_i}
= \lim_i \lim_j \ip{b}{\mu_i\star\lambda_j}. \]
We now use this to show that $\Phi_B$ is a $*$-homomorphism.  From the above,
\[ \ip{\Phi_B(b^*)}{\mu\otimes\lambda}
= \lim_i \lim_j \ip{b^*}{\mu_i\star\lambda_j}
= \lim_i \lim_j \overline{ \ip{b}{\mu_i^*\star\lambda_j^*} }
= \overline{\ip{\Phi(b)}{\mu^*\otimes\lambda^*}}
= \ip{\Phi_B(b)^*}{\mu\otimes\lambda}. \]
That $(\mu_i\star\lambda_j)^* = \mu_i^*\star\lambda_j^*$ follows as $\Delta_u$
is a $*$-homomorphism.  So $\Phi_B$ is a $*$-map.

We now proceed to check that $\Phi_B$ is multiplicative, for which we
``play-off'' $U$ and $U'$.
By construction, $(\id\otimes\omega)(U') = (\id\otimes\omega)(U) \in B$ for all
$\omega\in\mc B(H)_*$.  For $\mu,\lambda\in B^*$ choose $(\mu_i),(\lambda_i)$
as above, and let $\omega\in\mc B(H)_*$.  Then
\begin{align*} \ip{U'_\mu U'_\lambda}{\omega}
&= \ip{\mu}{(\id\otimes U'_\lambda\omega)(U)}
= \lim_i \ip{(\id\otimes U'_\lambda\omega)(U)}{\mu_i}
= \lim_i \ip{U'_\lambda}{\omega U_{\mu_i}} \\
&= \lim_i \ip{\lambda}{(\id\otimes \omega U_{\mu_i})(U)}
= \lim_i \lim_j \ip{U_{\lambda_j}}{\omega U_{\mu_i}}
= \lim_i \lim_j \ip{U_{\mu_i} U_{\lambda_j}}{\omega} \\
&= \lim_i \lim_j \ip{U_{\mu_i\star\lambda_j}}{\omega}
= \lim_i \lim_j \ip{(\id\otimes\omega)(U)}{\mu_i\star\lambda_j}
= \ip{\Phi_B((\id\otimes\omega)(U))}{\mu\otimes\lambda}.
\end{align*}
We conclude $\Phi_B((\id\otimes\omega)(U'))
= (\id\otimes\id\otimes\omega)(U'_{13} U'_{23})$.  Notice also that
\[ \ip{U'_\mu}{\omega} = \ip{\mu}{(\id\otimes\omega)(U')}
= \lim_i \ip{(\id\otimes\omega)(U)}{\mu_i}
= \lim_i \ip{U_{\mu_i}}{\omega}, \]
so $U_{\mu_i} \rightarrow U'_\mu$ weak$^*$ in $\mc B(H)$.

We now use the extended Haagerup tensor product, compare
Section~\ref{sec:haatenprod}.
Let $b = (\id\otimes\omega_{\xi,\eta})(U)\in B$, and let $(e_i)$ be an
orthonormal basis of $H$, so that, as in the previous paragraph
\begin{align*} \ip{\Phi_B(b)}{\mu\otimes\lambda}
&= \lim_i \lim_j \ip{U_{\mu_i} U_{\lambda_j}}{\omega}
= \lim_i \ip{U_{\mu_i} U'_\lambda}{\omega}
= \ip{U'_\mu U'_\lambda}{\omega} \\
&= \sum_k (U'_\mu e_k|\eta) (U'_\lambda\xi|e_k)
= \sum_k \ip{\mu}{(\id\otimes\omega_{e_k,\eta})(U)}
\ip{\lambda}{(\id\otimes\omega_{\xi,e_k})(U)}.
\end{align*}
Thus $\Phi_B(b) = \sum_i (\id\otimes\omega_{e_i,\eta})(U)
\otimes (\id\otimes\omega_{\xi,e_i})(U)$ a sum which converges in
$B\ehten B\subseteq B^{**}\ehten B^{**}$ treated as a (not closed) subspace
of $B^{**}\vnten B^{**}$.  However, we could alternatively treat $B$ as a
subalgebra of $M(C_0^u(\G))$ and in turn treat this as a subalgebra of
$C_0(\G)^{**}$.  Exactly the same argument then shows that
$\sum_i (\id\otimes\omega_{e_i,\eta})(U)
\otimes (\id\otimes\omega_{\xi,e_i})(U)$ converges weak$^*$ in
$C_0^u(\G)^{**} \ehten C_0^u(\G)^{**}$ and is equal to $\Delta_u(b)$,
identifying $M(C_0^u(\G)\otimes C_0^u(\G))$ with a subalgebra of
$C_0^u(\G)^{**} \vnten C_0^u(\G)^{**}$.

By \cite[Lemma~5.4]{er2} we know that the extended Haagerup tensor product
is ``injective''.  In particular, thinking of $B$ as a subalgebra of
$C_0^u(\G)^{**}$, the map $B\ehten B\rightarrow C_0^u(\G)^{**}\ehten
C_0^u(\G)^{**}$ is a complete isometry.  Using the arguments of
\cite[Sections~6-7]{er2} (a result proved for von Neumann algebras
in \cite{bsm}) we know that $B\ehten B$ is an algebra.  Thus, if we
set $X = \{ (\id\otimes\omega)(U) : \omega\in\mc B(H)_* \}$ then $X$ is
a dense subalgebra of $B$, and $\Phi_B$ restricts to a map $X\rightarrow
B\ehten B$.  Under the identifications derived above, $\Phi_B$ is equal to
$\Delta_u$ regarded as a map $X\rightarrow B\ehten B$.  As $\Delta_u$ is
homomorphism, so is $\Phi_B$, restricted to $X$.  By continuity,
$\Phi_B:B\rightarrow B^{**}\vnten B^{**}$ is a homomorphism, as required.

In summary, we have that $\mathbb S = (B,\Phi_B,U',H)$ is a $C^*$-Eberlein
algebra and the inclusion $\theta:B\rightarrow M(C_0^u(\G))$ is a non-degenerate
$*$-homomorphism.  By construction, $\theta^*:C_0^u(\G)^*\rightarrow B^*$ is
a homomorphism, and so $\theta$ defines a morphism $\mathbb S\rightarrow
\mathbb G$.
\end{proof}

\begin{definition}\label{def:embedded}
$C^*$-Eberlein algebras $(B,\Phi_B,U',H)$ which arise in this way are
\emph{embedded $C^*$-Eberlein algebras of $\G$}.
\end{definition}

\begin{remark}\label{rem:why_use_full_cstar}
Let us just remark that we used very little of the structure of locally
compact quantum groups here: the results would make sense of any $C^*$-bialgebra
$(A,\Delta)$ replacing $C_0^u(\G)$, with the appropriate notion of a ``morphism''.

In particular, the theory would work with $C_0^u(\G)$ replaced by $C_0(\G)$.
However, there is a subtle point here.  Any unitary corepresentation $U$ of
$C_0(\G)$ ``lifts'' uniquely to a corepresentation $V$ of $C_0^u(\G)$,
see \cite[Proposition~6.6]{ku}.  If $V$ satisfies the condition of
Theorem~\ref{thm:emb_cseb} above, then so does $U$, as we can just apply the
canonical surjection $C_0^u(\G) \rightarrow C_0(\G)$.  However, we have been
unable to decide if the converse to this claim holds.
\end{remark}

\section{Invariant means}\label{sec:means}

Classically, the existance of invariant means is proved by use of a fixed
point theorem (for example, the use of the Ryll-Nardzewski Theorem in the proof
of \cite[Theorem~2.15, Chapter~4]{bjm}).
Instead, we shall use a Hilbert space technique.

We proceed with some generality.  Let $A$ be a $C^*$-algebra with
$\Phi:A\rightarrow A^{**}\vnten A^{**}$ a unital $*$-homomorphism inducing a
Banach algebra product on $A^*$.  Let $U\in A^{**}\vnten\mc B(H)$ be a
corepresentation of $(A,\Phi)$, in the sense of Remark~\ref{rem:corep}.

\begin{definition}\label{defn:invvec}
A vector $\xi\in H$ is \emph{invariant} for $U$ if $(\mu\otimes\iota)(U)
\xi = \xi \ip{\mu}{1}$ for all $\mu\in A^*$.
Denote by $\inv(U)$ the collection of all invariant vectors of $U$.

For $\xi\in H$, the \emph{average} of $\xi$ is the unique vector of minimal
norm in the closure of the convex set $C_\xi := \{ (\mu\otimes\iota)(U)\xi
: \mu\text{ a state on }A \}$, which exists by properties of Hilbert spaces.
\end{definition}

\begin{proposition}\label{prop:average}
Let $U$ be a contractive corepresentation of $A$ on $H$, and let $\xi\in H$.
The average of $\xi$ is invariant for $U$.
\end{proposition}
\begin{proof}
Let $\xi_0$ be the average of $\xi$.  If $\mu,\lambda\in A^*$ are states,
then so is $\mu\star\lambda$, as $\Phi$ is a unital $*$-homomorphism.  Hence
$(\mu\otimes\id)(U) (\lambda\otimes\id)(U) \xi =
(\mu\star\lambda\otimes\id)(U)\xi \in C_\xi$.  It follows by continuity that
$(\mu\otimes\id)(U)$ leaves $C_\xi$ invariant.  As $U$ is contractive, if
$\mu\in A^*$ is a state, then $(\mu\otimes\id)(U)$ is a contractive operator.
Hence $(\mu\otimes\id)(U)\xi_0 \in C_\xi$ and $\| (\mu\otimes\id)(U)\xi_0 \|
\leq \|\xi_0\|$.  By uniqueness, $(\mu\otimes\id)(U)\xi_0 = \xi_0$.
This holds for all states, so by polar decomposition, $(\mu\otimes\id)(U)\xi_0
= \ip{\mu}{1} \xi_0$ for all $\mu\in A^*$, as claimed.
\end{proof}

We now proceed to construct means on C$^*$-Eberlein algebras.
We are inspired by \cite[Chapter~2.3]{bjm}, but we shall work with the language
of Banach algebras.  Let $\mathbb S=(A,\Phi,V,H)$ be a $C^*$-Eberlein algebra,
where in this section we shall suppose that $A$ is unital.
Recall from Proposition~\ref{prop:get_dba} that $A^*$ becomes a dual Banach
algebra, and so $A$ is an $A^*$-bimodule.

\begin{definition}\label{defn:invmean}
A state $M\in A^*$ is a \emph{left invariant mean} if
$\ip{M}{a\cdot\mu} = \ip{M}{a}$ for all states $\mu\in A^*$ and $a\in A$.
Similarly we defined a \emph{right invariant mean}.  A mean which is both
left and right invariant is an \emph{invariant mean}.
\end{definition}

As $\Phi:A\rightarrow A^{**}\vnten A^{**}$ we can make sense of the map
$(\id\otimes M)\Phi:A\rightarrow A^{**}$ (which will be completely positive).
Then $M$ is left invariant if and only if $(\id\otimes M)\Phi$ is the rank-one
map $a\mapsto \ip{M}{a} 1$; similarly $M$ is right invariant if 
$(M\otimes \id)\Phi$ is the rank-one map $a\mapsto \ip{M}{a} 1$.

\begin{theorem}\label{thm:wheninvmean}
For $a\in A$, let $K(a) = \{ \mu\cdot a : \mu \text{ is a state on }A \}$.
Then the following are equivalent:
\begin{enumerate}
\item\label{thm:wheninvmean:one}
$A$ has a left invariant mean;
\item\label{thm:wheninvmean:two}
$K(a)\cap \mathbb C1 \not= \emptyset$ for all $a\in A$;
\item\label{thm:wheninvmean:three}
$0\in K(a-a\cdot\mu)$ for all $a\in A$ and all states $\mu\in A^*$.
\end{enumerate}
\end{theorem}
\begin{proof}
This is an abstract version of \cite[Theorem~2.3.11]{bjm}.  By definition,
$M\in A^*$ is left invariant when $\mu\star M=\ip{\mu}{1} M$ for all
$\mu\in A^*, a\in A$, or equivalently, if $M\cdot a = \ip{M}{a} 1$ for all
$a\in A$.
So (\ref{thm:wheninvmean:one})$\Rightarrow$(\ref{thm:wheninvmean:two}).
If (\ref{thm:wheninvmean:two}) holds then for $a\in A$ there is a state
$\mu\in A^*$ and $t\in\mathbb C$ with $\mu\cdot a=t1$.  Then for a state
$\lambda\in A^*$ we have that $\mu\cdot(a-a\cdot\lambda) =
t1 - (\mu\cdot a)\cdot\lambda = t1-t1\cdot\lambda = 0$ as $1\cdot\lambda
= (\lambda\otimes\id)\Phi(1)=1$.
So (\ref{thm:wheninvmean:two})$\Rightarrow$(\ref{thm:wheninvmean:three}).

Now suppose that (\ref{thm:wheninvmean:three}) holds.  For each $a\in A$ and
state $\mu$ let
\[ M(a,\mu)= \{ \lambda\in A^*\text{ a state } : \lambda\cdot
(a-a\cdot\mu) = 0 \}, \]
which is non-empty by assumption.  If $(\lambda_\alpha)$ is a net in $M(a,\mu)$
converging weak$^*$ to $\lambda$, then $\lambda$ is a state, and for any
$\phi\in A^*$,
\[ \ip{\phi}{\lambda\cdot(a-a\cdot\mu)}
= \ip{\lambda}{(a-a\cdot\mu)\cdot\phi}
= \lim_\alpha \ip{\lambda_\alpha}{(a-a\cdot\mu)\cdot\phi}
= \lim_\alpha \ip{\phi}{\lambda_\alpha\cdot(a-a\cdot\mu)} = 0. \]
So $\lambda\in M(a,\mu)$ and we conclude that $M(a,\mu)$ is weak$^*$-closed.

We claim that the family $\{ M(a,\mu) : a\in A, \mu \text{ a state} \}$ has the
finite intersection property.  If so, then as the unit ball of $A^*$ is
weak$^*$-compact, there is $\lambda\in M(a,\mu)$ for all $a,\mu$.  Set $M=
\lambda\star\lambda$ so
\[ \ip{M}{a\cdot\mu} = \ip{\lambda}{\lambda\cdot(a\cdot\mu)}
= \ip{\lambda}{\lambda\cdot a} = \ip{M}{a}, \]
that is, $M$ is left invariant.  As $\Phi$ is a unital $*$-homomorphism,
and $\lambda$ is a state, also $M$ is a state as required to show
(\ref{thm:wheninvmean:one}).

To show the finite intersection property, we use induction.  Let
$a_1,\cdots,a_n\in A$ and $\mu_1,\cdots,\mu_n$ be states on $A$, and
suppose that $\lambda\in \bigcap_{j=1}^{n-1} M(a_j,\mu_j)$.  As
(\ref{thm:wheninvmean:three}) holds, we can find
$\phi \in M(\lambda\cdot a_n,\mu_n)$, so
\[ 0 = \phi\cdot( \lambda\cdot a_n  - \lambda\cdot a_n\cdot\mu_n)
= (\phi\star\lambda)\cdot(a_n-a_n\cdot\mu_n). \]
However, for $1\leq j<n$,
\[ (\phi\star\lambda)\cdot (a_j-a_j\cdot\mu_j)
= \phi\cdot\big( \lambda\cdot(a_j-a_j\cdot\mu_j) \big)
= \phi\cdot 0 = 0. \]
Thus $\phi\star\lambda\in\bigcap_{j=1}^{n} M(a_j,\mu_j)$, and so the result
follows by induction.
\end{proof}

\begin{remark}
If $M$ is a left invariant mean on $A$ then for any state $\mu\in A^*$ and
$a\in A$, we have that $(M\star\mu)\cdot a = M \cdot(\mu\cdot a)
= \ip{M}{\mu\cdot a} 1 = \ip{M\star\mu}{a}1$ and so $M\star\mu$ is also
a left invariant mean.  Thus, if $t1\in K(a)$, that is, there is a state
$\mu$ with $\mu\cdot a=t1$, then $(M\star\mu)\cdot a = M\cdot(\mu\cdot a)
= t1$.  Thus $K(a) \cap \mathbb C1 = \{ \lambda\cdot a : \lambda
\text{ a left invariant mean} \}$.
\end{remark}

\begin{theorem}
Let $(A,\Phi,V,H)$ be a C$^*$-Eberlein algebra with $A$ unital.
Then $A$ has a left invariant mean.
\end{theorem}
\begin{proof}
We verify condition (\ref{thm:wheninvmean:two}) of the previous theorem.
Indeed, we shall show that
$K(a)\cap\mathbb C1\not=\emptyset$ for $a\in A_V$.  This suffices, for if
$a\in A$ there is a sequence $(a_n)$ in $A_V$ converging to $a$.
By assumption, for each $n$ there is a state $\mu_n$ and $t_n\in\mathbb C$
with $\mu_n \cdot a_n = t_n 1$.  Notice that $|t_n| \leq \|\mu_n\|\|a_n\|
= \|a_n\|\rightarrow \|a\|$.  So by moving to subnets, we may suppose that
$t_n\rightarrow t\in\mathbb C$ and $\mu_n\rightarrow\mu$ weak$^*$.  Then,
for $\lambda\in A^*$,
\begin{align*} |\ip{\lambda}{\mu_n\cdot a_n} - \ip{\lambda}{\mu\cdot a}|
&= | \ip{\mu_n}{a_n\cdot\lambda} - \ip{\mu}{a\cdot\lambda}| \\
&\leq | \ip{\mu_n}{a_n\cdot\lambda} - \ip{\mu_n}{a\cdot\lambda} |
  + | \ip{\mu_n}{a\cdot\lambda} - \ip{\mu}{a\cdot\lambda} | \\
&= | \ip{\mu_n}{a_n\cdot\lambda - a\cdot\lambda} |
  + | \ip{\mu_n - \mu}{a\cdot\lambda} |,
\end{align*}
which converges to $0$, as $a_n\cdot\lambda - a\cdot\lambda\rightarrow 0$
in norm, and $\mu_n-\mu\rightarrow 0$ weak$^*$.  So $t_n 1 = \mu_n\cdot a_n
\rightarrow \mu\cdot a$ weakly, hence $\mu\cdot a = t1$.  So $K(a)\cap
\mathbb C1\not=\emptyset$ for all $a$.

By replacing $V$ by $V_{12}$ acting on $H\otimes\ell^2$, we may suppose that
$A_V = \{ (\id \otimes\omega_{\xi,\eta})(V) : \xi,\eta\in H \}$.  So let
$a=(\id\otimes\omega_{\xi,\eta})(V)$.  As in the proof of
Proposition~\ref{prop:get_dba}, $\mu\cdot a=(\id\otimes
V_\mu\omega_{\xi,\eta})(V) = (\id\otimes\omega_{V_\mu\xi,\eta})(V)$
for $\mu\in A^*$.  As $V$ is unitary, by
Proposition~\ref{prop:average}, if $\xi_0$ is the average of $\xi$,
then $\xi_0 \in \inv(V)$ and $\xi_0$ is in the
closure of $\{ V_\mu\xi : \mu\text{ a state}\}$.  So there is a net of
states $(\mu_i)$ with $V_{\mu_i} \xi \rightarrow \xi_0$ in norm.  Passing to
a subnet, we may suppose that $\mu_i\rightarrow\mu$ weak$^*$.  So
\[ \mu\cdot a = (\id\otimes\omega_{V_\mu\xi,\eta})(V)
= \lim_i (\id\otimes\omega_{V_{\mu_i}\xi,\eta})(V)
= (\id\otimes\omega_{\xi_0,\eta})(V) = (\xi_0|\eta)1, \]
as $\xi_0$ is invariant.  Hence $(\xi_0|\eta)1 \in K(a)$, as required.
\end{proof}

\begin{theorem}\label{thm:invmeanuni}
Let $(A,\Phi,V,H)$ be a C$^*$-Eberlein algebra with $A$ unital.
Let $M\in A^*$ be a left invariant mean.  Then $M$ is the unique left
invariant mean, and $(M\otimes\id)(V)\in\mc B(H)$
is the orthogonal projection onto $\inv(V)$.  If $V$ is unitary, then
$M$ is also right invariant.
\end{theorem}
\begin{proof}
Let $p=(M\otimes\id)(V)\in\mc B(H)$.  For $a=(\id\otimes\omega)(V) \in A$ we
have that $\ip{M}{a}1 = M\cdot a = (\id\otimes p\omega)(V)$.  So with
$\omega=\omega_{\xi,\eta}$ we see that for any $\mu\in A^*$,
\[ (p\xi|\eta) \ip{\mu}{1} = \ip{M\otimes\omega}{V} \ip{\mu}{1}
= \ip{M}{a} \ip{\mu}{1} = \ip{\mu}{(\id\otimes p\omega)(V)}
= (V_\mu p \xi|\eta), \]
so $V_\mu p\xi = \ip{\mu}{1} p\xi$ for all $\mu\in A^*$, that is,
$p\xi\in\inv(V)$.

Conversely, if $\xi\in\inv(V)$ then $V_\mu\xi = \ip{\mu}{1}\xi$ for all
$\xi\in A^*$, and so in particular, $p\xi = V_M\xi = \xi$ as $M$ is a state.
So $p$ is an idempotent with image $\inv(V)$.  As $p$ is contractive, $p$ must
be the orthogonal projection onto $\inv(V)$.

If $N$ is also a left invariant mean then also $(N\otimes\id)(V)=p$.
So for $a=(\id\otimes\omega)(V)\in A_V$ we find that $\ip{N}{a} =
\ip{p}{\omega} = \ip{M}{a}$.  By density of $A_V$ in $A$, it follows that
$M=N$.

That $p\xi \in \inv(V)$ for all $\xi$ can be equivalently stated as
$V(1\otimes p) = 1\otimes p$.  If $V$ is unitary then also $V^*(1\otimes p)
= 1\otimes p$ and so $(1\otimes p)V = 1\otimes p$.  Then, for
$a=(\id\otimes\omega)(V)\in A_V$, we see that
\[ a\cdot M = (\id\otimes\omega p)(V) = (\id\otimes\omega)((1\otimes p)V)
= (\id\otimes\omega)(1\otimes p) = \ip{M}{a} 1. \]
By continuity this holds for all $a\in A$ and so $M$ is right invariant.
\end{proof}

\begin{remark}
Let $\sigma:A^{**}\vnten A^{**} \rightarrow A^{**}\vnten A^{**}$
be the swap map, so $\sigma\Phi$ induces the opposite multiplication on $A^*$.
As $A_V$ is norm dense in $A$, which is a $C^*$-algebra, it follows that
$A_{V^*}$ and $A_V$ have the same closure (which is $A$).  Then $V^*$ will
be a corepresentation for $\sigma\Phi$, and so we conclude that
$(A,\sigma\Phi,V^*,H)$ will be a $C^*$-Eberlein algebra.

In this way, we can swap the roles of left and right invariant means.  So the
previous theorem also shows that right invariant means are unique, and if $V$
is unitary, right invariant means are invariant.
\end{remark}

\section{Compactifications}\label{sec:cmpts}

We take a category theory approach to compactifications, as follows.
By a ``compactification'' of $G\in\lcg$ we mean the universal object in
one of $\textsf{CSCH}$, $\textsf{CWCH}$, $\textsf{CSTS}$, $\textsf{CG}$.
Letting $\textsf{Cat}$ be one of these categories, this means that we seek 
$K\in \textsf{Cat}$ and a morphism $G\rightarrow K$ with the universal
property that for any $H\in\textsf{Cat}$ and any morphism $G\rightarrow H$,
there is a unique morphism $K\rightarrow H$ with the following diagram
commuting:
\[ \xymatrix{ G \ar[r] \ar[rd] & K \ar[d]^{!} \\ & H } \]
For example, as explained in \cite[Section~2]{daws}, one can think of the
\emph{Bohr} or \emph{almost periodic} compactification of $G$ as being
a universal object in $\textsf{CG}$, the category of compact groups.
This is just the ``largest'' compact group to contain a dense
homomorphic copy of our starting semigroup.  Working instead in
$\textsf{CSTS}$ we obtain the \emph{weakly almost periodic} compactification.

Let us think about what happens in $\textsf{CWCH}$ in the context of \cite{ss}.
Well, \cite[Theorem~2.17]{ss} shows that for $G\in\lcg$ there is 
$G^{\mc CH}\in\textsf{CWCH}$ and a morphism
$\epsilon_{\mc CH} : G \rightarrow G^{\mc CH}$ with dense range,
such that for any $S\in\textsf{CWCH}$ with
$\phi:G\rightarrow S$ having dense range, there is
$\theta:G^{\mc CH}\rightarrow S$ with $\theta\circ\epsilon_{\mc CH}=\phi$;
compare the definitions in \cite[Section~2.1]{ss}.  This is also a
compactification in our sense:
if $S\in\textsf{CWCH}$ and $\phi:G\rightarrow S$ is any morphism, then set
$S'$ to be the closure of the image of $\phi$, so also $S'\in\textsf{CWCH}$.
Let $\phi':G\rightarrow S'$
be the corestriction of $\phi$.  So there is
$\theta':G^{\mc CH}\rightarrow S'$, which is surjective, with
$\theta'\circ\epsilon^{\mc CH} = \phi'$.  Let $\theta$ be the composition of
$\theta'$ with the inclusion $S'\rightarrow S$, so
$\theta\circ\epsilon^{\mc CH} = \phi$.  As $\epsilon^{\mc CH}$ has dense range,
it is clear that $\theta$ is unique, as required.

We now wish to construct compactifications in our non-commutative setting,
using the notion of morphisms considered in the previous section.
Recall the notion of an embedded $C^*$-Eberlein algebra from
Definition~\ref{def:embedded}.  Such $C^*$-Eberlein algebras biject with
unitary corepresentations $U$ of $C_0^u(\G)$ such that the norm closure of
$\{ (\id\otimes\omega)(U) : \omega\in\mc B(H)_* \}$ is a C$^*$-algebra,
say $A$.

\begin{theorem}\label{thm:comp_construction}
Let $\G$ be a locally compact quantum group.  There is an (essentially unique)
unital embedded $C^*$-Eberlein algebra $(A,\Phi_A,U_A,H_A)$ such that for any
other unital embedded $C^*$-Eberlein algebra $(B,\Phi,U,H)$, we have that
$B\subseteq A$ with the inclusion $B\rightarrow A$ being a morphism.
\end{theorem}
\begin{proof}
We claim that it is enough to find an embedded C$^*$-Eberlein $A$ such that
any other embedded C$^*$-Eberlein algebra $B$ is contained in $A$.  Indeed,
if so, then the inclusion $B\rightarrow A$ is non-degenerate, as it is unital.
As both $A$ and $B$ are embedded, the products on $A^*$ and $B^*$ are
induced by $\Delta_u$, and hence we immediately see that $A^*\rightarrow B^*$
is a homomorphism, so $B\rightarrow A$ is a morphism.

We now show the existence of $A$.  Let
\[ X = \{ (\id\otimes\omega)(U_B) : (B,\Phi_B,U_B,H_B) \text{ an embedded
C$^*$-Eberlein algebra}\}, \]
a subset of $M(C_0^u(\G))$.  For each $x\in X$ let $U_x$ be a choice of suitable
corepresentation on $H_x$, and $\omega_x\in\mc B(H_x)_*$, such that
$x = (\id\otimes\omega_x)(U_x)$.  Let $A$ be the C$^*$-algebra generated by
slices of $\{U_x:x\in X\}$, and let $U$ be the unitary corepresentation
\[ \bigoplus_{x\in X} U_x \oplus \bigoplus_{x\not=y}
U_x \tp U_y \oplus \bigoplus_{n\geq 3} \bigoplus_{x_1\not=x_2\not=\cdots\not=
x_n} U_{x_1}\tp\cdots\tp U_{x_n}, \]
acting on $H$ say.  This is a sort of ``free-product'' construction.
By slicing $U$ we generate elements in $A$; for example, looking at the second term,
we get elements (amoungst others) of the form $ (\id\otimes\omega_1)(U_x)
(\id\otimes\omega_2)(U_y) $ for $x\not=y$.
It follows that slices of $U$ will be norm dense in $A$ and hence $A$ is
an embedded C$^*$-Eberlein algebra.

By construction, for any embedded C$^*$-Eberlein algebra $(B,\Phi_B,U_B,H_B)$
and any $\omega\in\mc B(H_B)_*$, we have that $x=(\id\otimes\omega)(U_B)\in X$
and so in particular $x\in A$.  As such $x$ are dense in $B$, it follows that
$B\subseteq A$, as required.
\end{proof}

Let $E(\G) = (\mc E(\G), \Phi_\G, U_\G, H_\G)$ be the embedded $C^*$-Eberlein
algebra given by the previous theorem.  Notice that $\mc E(\G),\Phi_\G$ are
unique, but there is some, unimportant, choice in $U_\G,H_\G$.  There is a
further notational point: by definition, $U_\G$ is a ``corepresentation''
of $\mc E(\G)$, and so $U_\G \in \mc E(\G)^{**} \vnten \mc B(H_\G)$.  The
construction shows that $U_\G$ arises from Proposition~\ref{prop:subalg}
applied to a unitary corepresentation of $C_0^u(\G)$.  We shall abuse
notation, and also write $U_\G$ for this unitary in
$M(C_0^u(\G)\otimes\mc B_0(H_\G))$.

\begin{theorem}
$E(\G)$ is a compactification of $\G$ in the category of C$^*$-Eberlein
algebras.  That is, given any $C^*$-Eberlein algebra $\mathbb S = 
(B,\Phi_B,U_B,H_B)$ with $B$ unital, and any morphism $\G\rightarrow\mathbb S$,
there is a unique morphism $E(\G)\rightarrow\mathbb S$ making the diagram
commute:
\[ \xymatrix{ \G \ar[r] \ar[rd] & \mathbb S \\
& E(\G). \ar[u]_{!} } \]
\end{theorem}
\begin{proof}
Let $\G\rightarrow\mathbb S$ be induced by $\theta:B\rightarrow M(C_0^u(\G))$.
Then $\theta(B)$ is an embedded $C^*$-Eberlein algebra and so $\theta(B)
\subseteq\mc E(\G)$.  Hence we may let $\theta_0$ be the corestriction of
$\theta$ to a map $B\rightarrow\mc E(\G)$.  As a unital $*$-homomorphism,
this is non-degenerate, and is hence a morphism; clearly the resulting
diagram commutes:
\[ \xymatrix{ M(C_0^u(\G)) & B \ar[l]_-{\theta} \ar[d]^{\theta_0} \\
& \mc E(\G). \ar@{_{(}->}[lu] } \]
If $\phi:B\rightarrow\mc E(\G)$ is another morphism making the diagram
commute, then as the inclusion $j:\mc E(\G)\rightarrow M(C_0^u(\G))$ is
injective, we see that as $j\theta_0 = \theta = j\phi$, also $\theta_0=\phi$.
Hence $\theta_0$ is unique.
\end{proof}

For any embedded $C^*$-Eberlein algebra, the generating corepresentation is
unitary, and we can apply the results of Section~\ref{sec:means} to show
the following.

\begin{theorem}
For any $\G$ the $C^*$-Eberlein algebra $E(\G)$ admits a unique invariant
mean.
\end{theorem}

Thanks to Remark~\ref{rem:lcqg_is_cstareb}, we know that $C_0^u(\G)\subseteq
\mc E(\G)$.  The following is again known in the classical situation,
for example \cite[Corollary~3.7]{ber}.

\begin{proposition}
Let $\G$ be a locally compact, non-compact, quantum group.
Let $M$ be the unique invariant mean on $\mc E(\G)$.  Then $M$ annihilates
$C_0^u(\G)$.
\end{proposition}
\begin{proof}
Suppose not, so if $\varphi$ is $M$ restricted to $C_0^u(\G)$, then $\varphi$
is a non-zero state.  That $M$ is left-invariant, and that the inclusion
$C_0^u(\G) \rightarrow \mc E(\G)$ is a morphism, implies that $\mu\star\varphi
= \varphi \ip{\mu}{1}$ for all $\mu\in C_0^u(\G)^*$.

Now, by Cohen Factorisation, \cite[Proposition~A.2]{mnw}, there is
$a\in C_0^u(\G)$ and $\phi\in C_0^u(\G)^*$ with $\varphi = \phi a$.  For
$x\in C_0^u(\G)$ we know that $(1\otimes a)\Delta(x) \in C_0^u(\G)\otimes
C_0^u(\G)$, see \cite[Proposition~6.1]{ku}, and so
\[ (\id\otimes\phi)\big( (1\otimes a)\Delta(x) \big) \in C_0^u(\G). \]
Then, for $\mu\in C_0^u(\G)^*$, by left-invariance,
\[ \ip{\mu}{1} \varphi(x) = \ip{\mu\star\varphi}{x} =
\ip{\mu\otimes\phi}{(1\otimes a)\Delta(x)}. \]
As $\mu$ was arbitrary, this implies that $\varphi(x) 1 =
(\id\otimes\phi)\big( (1\otimes a)\Delta(x) \big) \in C_0^u(\G)$ and so by
a suitable choice of $x$ we conclude that $1\in C_0^u(\G)$
and so $\G$ is compact.
\end{proof}

\section{The Quantum Bohr compactification}\label{sec:bohr}

For a locally compact group $G$, we know that $\wap(G)$ splits as a direct
sum of the almost periodic functions $\ap(G)$ and those $f\in\wap(G)$ such
that $|f|^2$ is annihilated by the invariant mean; see \cite[Theorem~5.11]{dg}
or for a more semigroup-theoretic approach, see \cite[Theorem~3.11, Chapter~4]{bjm}.
For earlier results about positive definite functions (that is, in our
language, about the Eberlein compactification) were shown by Godement in
\cite{G}, see also \cite[Addenda, 16.5.2]{dix}.  Over the next two sections,
we develop analogous results for non-commutative $C^*$-Eberlein algebras,
starting with an investigation of the quantum analogue of almost periodic
functions.

In \cite{soltan}, So{\l}tan explored a similar construction to that in the
above section, but for compact quantum groups.  We recall, \cite{woro2,mvd},
that a compact quantum group is a unital $C^*$-bialgebra $(K,\Delta)$ with
the ``cancellation rules'', namely that the following sets are both linearly
dense in $A\otimes A$:
\[ \{ (a\otimes 1)\Delta(b) : a,b\in K \}, \qquad
\{ (1\otimes a)\Delta(b) : a,b\in K \}. \]
Then, for any $C^*$-bialgebra $(A,\Delta_A)$ So{\l}tan shows how to construct
a compact quantum group $(K,\Delta_K)$ and a morphism (here meaning a
non-degenerate $*$-homomorphism intertwining the coproducts) $K\rightarrow
M(A)$ such that for any compact quantum group $(B,\Delta_B)$ and any morphism
$B\rightarrow M(A)$, there is a morphism $B\rightarrow K$ making the three
morphisms commute.  Again, the construction realises $K$ as a subalgebra of
$M(A)$ and $\Delta_K$ as the restriction of $\Delta_A$.

Let $U = (U_{ij}) \in \mathbb M_n(M(A)) = M(A\otimes\mathbb M_n)$ be a
finite-dimensional (unitary, or just invertible) corepresentation of
$(A,\Delta_A)$.  We say that $U$ is admissible if the ``transpose''
$(U_{ji})$ is also invertible in the algebra $\mathbb M_n(M(A))$.  Then we can
construct $K$ as the closed linear span of elements of the form $U_{ij}$
as $U$ varies over all admissible corepresentations.

One of the problems of So{\l}tan's construction, further explored in
\cite{daws}, is that it might be hard to decide if a given finite-dimensional
unitary corepresentation is admissible (although we note that there are no
known counter-examples to the conjecture that any finite-dimensional unitary
corepresentation of a locally compact quantum group is already admissible.)
The following results are interesting from this perspective.

The second named author explored this construction more in \cite{daws}.
In particular, we introduced the notation that, when So{\l}tan's construction
is applied to $C_0^u(\G)$, we get the compact quantum group 
$(\mathbb{AP}(\G),\Delta_{\mathbb{AP}(\G)})$.  So $\mathbb{AP}(\G) \subseteq
M(C_0^u(\G))$ is the ``maximal'' compact quantum group inside $M(C_0^u(\G))$.
We remark that if we work instead with $C_0(\G)$, we may obtain a (possibly)
different compact quantum group, but the underlying Hopf $*$-algebra is the same.

\begin{proposition}\label{prop:haarstate}
For any $\G$ we have that $\mathbb{AP}(\G) \subseteq \mc E(\G)$, and the
mean on $\mc E(\G)$ restricts to the Haar state on $\mathbb{AP}(\G)$.
\end{proposition}
\begin{proof}
Set $A=\mathbb{AP}(\G) \subseteq M(C_0^u(\G))$.  Let
$W\in M(A\otimes\mc B_0(H))$ be the left-regular representation for the
compact quantum group $A$.   Thus slices of $W$ are norm dense in $A$, and it
is easy to verify that $(A,\Delta_u|_A,W,H)$ is a $C^*$-Eberlein algebra.
So we immediately see that $A\subseteq \mc E(\G)$.
The restriction of the mean on $\mc E(\G)$ to $\mathbb{AP}(\G)$ will be a
bi-invariant state, and hence by uniqueness, is the Haar state.
\end{proof}

\begin{proposition}
Let $V$ be a finite-dimensional sub-corepresentation of $U_\G$.  Then
$V$ is admissible.
\end{proposition}
\begin{proof}
This is similar to the argument for compact quantum groups, compare for
example \cite[Proposition~6.10]{mvd}, and so we shall not give all the details.
By assumption, there is a finite-dimensional subspace $K\subseteq H_\G$ such
that, with $q:H_\G\rightarrow K$ the orthogonal projection, we have that
$U_\G(1\otimes q) = (1\otimes q)U_\G$ and $V=(1\otimes q)U_\G(1\otimes q)
\in M(C_0^u(\G)\otimes\mc B(K))$.  If we fix an orthonormal basis
$(e_i)_{i=1}^n$ for $K$, then we can view $V$ as a matrix $(V_{ij}) \in
\mathbb M_n(M(C_0^u(\G)))$.  Then $V$ is admissible if $(V_{ji})$ is
invertible, or equivalently, if $\overline{V} = (V_{ij}^*)$ is invertible.
We shall show that actually $\overline{V}$ is similar to a unitary.

As $V$ is finite-dimensional and unitary, $V$ is the direct-sum of finitely
many irreducible unitary corepresentations.  So it suffices to prove the
claim when $V$ is irreducible, which we now assume.  As shown in
\cite[Lemma~6.9]{mvd}, for example, $\overline{V}$ will also be irreducible.
As $(\id\otimes\omega)(V) \in\mc E(\G)$ for all $\omega\in\mc B(K)_*$,
we can apply Proposition~\ref{prop:subalg} to form $V'\in\mc{E}(\G)^{**}
\vnten\mc B(K)$ (actually, as $K$ is finite-dimensional, it is both easy
to construct $V'$, and we find that actually $V'\in\mc E(\G)\otimes\mc B(K)$).
Similarly let $U_\G'$ be the copy of $U_\G$ in $\mc E(\G)^{**}\vnten
\mc B(H_\G)$. It is easy to verify that $V' = (1\otimes q)U_\G'(1\otimes q)$.

Let $M\in\mc E(\G)^*$ be the invariant mean so, as after
Definition~\ref{defn:invmean}, we have that $(M\otimes\id)\tilde\Phi_\G(x)
= \ip{x}{M}1$ for $x\in\mc E(\G)^{**}$.  We have the isomorphism
\[ \mc B(H_\G) \cong \begin{pmatrix} \mc B(K) & \mc B(K^\perp,K) \\
\mc B(K,K^\perp) & \mc B(K^\perp) \end{pmatrix}, \]
and similarly we can consider $\mc E(\G)^{**}\vnten\mc B(H_\G)$ as a matrix
algebra.  Let $x\in\mc B(K,H_\G)$, and consider $1\otimes x \in
\mc E(\G)^{**}\vnten\mc B(K,H_\G)$ regarded as the ``left column'' of our
matrix algebra.  This viewpoint allows us to define
\[ y = (M\otimes\id)\big( {U_\G'}^* (1\otimes x) \overline{V}' \big)
\in \mc B(K,H_\G). \]
As $\Phi$ is a $*$-homomorphism, it is easy to see that $\overline{V}$
will also be a corepresentation, and so
\[ (\tilde\Phi_\G\otimes\id)\big( {U_\G'}^* (1\otimes x) \overline{V}' \big)
= {U_\G'}^*_{23} {U_\G'}^*_{13} (1\otimes x) \overline{V}'_{13}
   \overline{V}'_{23}. \]
Applying $ (M\otimes\id\otimes\id)$ we find that
\[ 1\otimes y = 1\otimes (M\otimes\id)\big( {U_\G'}^* (1\otimes x)
   \overline{V}' \big)
= {U_\G'}^* (1\otimes y) \overline{V}'. \]
Thus $U_\G' (1\otimes y) = (1\otimes y) \overline{V}'$.  As $\overline{V}'$
is irreducible, if $y\not=0$, in a standard way this shows that
$\overline{V}'$ is equivalent to a sub-corepresentation of $U_\G'$, and
hence is similar to a unitary, showing the same for $\overline{V}$ as required.

So we need to show that for some choice of $x$, we get a non-zero $y$.
Towards a contradiction suppose not, so $y=0$ for all choices of $x$.
For $\xi\in H_\G, \eta\in K$ let $x = \theta_{\xi,\eta}$ a rank-one operator
$K\rightarrow H_\G$.  Then, for $\alpha\in K, \beta\in H_\G$,
\[ 0 = (y\alpha|\beta) = \ip{{U_\G'}^* (1\otimes \theta_{\xi,\eta})\overline{V}'}
   {M\otimes\omega_{\alpha,\beta}}
= \ip{M}{(\id\otimes\omega_{\xi,\beta})(U'_\G)
   (\id\otimes\omega_{\alpha,\eta})(\overline{V}')}. \]
Now, elements of the form $(\id\otimes\omega_{\xi,\beta})(U'_\G)$ are linearly
dense in $\mc E(\G)$, and $(\id\otimes\omega_{\alpha,\eta})(\overline{V}')
= (\id\otimes\omega_{\alpha,\eta})(V')$ is a linear combination of the
elements $V_{ij}^* \in \mc E(\G)$.  Thus $\ip{M}{a V_{ij}^*}=0$ for all
$a\in\mc E(\G)$, in particular, $\ip{M}{V_{ij} V_{ij}^*}=0$.  As $\ip{M}{1}=1$
and $V$ is unitary, this leads to the required contradiction.
\end{proof}

\subsection{Compact vectors}\label{sec:cmptvecs}

In this section, we generalise the classical notion of a \emph{compact vector}
for a unitary group representation. Classically, see \cite[Definition~1.7,
Lemma~1.8]{br} if $\pi$ is a unitary representation of a locally compact group
$G$ on a Hilbert space $H$, then $\xi\in H$ is \emph{compact} if and only if
$\{ \pi(g)\xi : g\in G \}$ is relatively compact in $H$, if and only if
$g\mapsto (\pi(g)\xi|\xi)$ is an almost periodic function.  The following is
hence a slightly stronger definition, but it is appropriate for our needs.

\begin{definition}
Let $\G$ be a locally compact group, and let $U\in M(C_0^u(\G)\otimes\mc B_0(H))$
be a unitary corepresentation.  Define $\xi\in H$ to be a \emph{compact vector} if
\[ (\id\otimes\omega_{\xi,\eta})(U) \in \mathbb{AP}(C_0^u(\G))
\qquad (\eta\in H). \]
It is easy to see that the collection of all compact vectors forms a closed
subspace of $H$, say $H_c$.
\end{definition}

From now on, fix a unitary corepresentation $U$ and let $H_c$ be the subspace
of compact vectors.

\begin{lemma}\label{lem:rest_com_vecs}
The corepresentation $U$ restricts to $H_c$ forming a corepresentation
$V\in M(C_0^u(\G)\otimes\mc B_0(H_c))$ say.  As $U$ is unitary, $U$ also
restricts to $H_c^\perp$ forming a corepresentation $V_0$ say.  Then
$V_0$ has no non-zero compact vectors.
\end{lemma}
\begin{proof}
We wish to show that for $\mu\in C_0^u(\G)^*$ and $\xi\in H_c$, we have that
$(\mu\otimes\id)(U)\xi \in H_c$.  Set $U_\mu = (\mu\otimes\id)(U)$, so for
$\eta\in H$, we see that
\[ (\id\otimes \omega_{U_\mu \xi,\eta})(U)
= (\id\otimes \omega_{\xi,\eta})(U(1\otimes U_\mu))
= (\id\otimes\mu\otimes\omega_{\xi,\eta})(U_{13} U_{23})
= (\id\otimes\mu)\Delta\big( (\id\otimes\omega_{\xi,\eta})(U) \big). \]
As $(\id\otimes\omega_{\xi,\eta})(U)\in\mathbb{AP}(C_0^u(G))$ we see that
\[ \Delta\big( (\id\otimes\omega_{\xi,\eta})(U) \big)  \in 
\mathbb{AP}(C_0^u(G)) \otimes \mathbb{AP}(C_0^u(G)), \]
the minimal tensor product, and so slices are also in
$\mathbb{AP}(C_0^u(G))$ as required to show that $U_\mu\xi\in H_c$ as claimed.

We next show that $V_0$ has no non-zero compact vectors.  Indeed, if
$\xi\in H_c^\perp$ is non-zero then $\xi$ is not compact for $U$, and so there
is $\eta = \eta_c + \eta_0 \in H_c \oplus H_c^\perp$ with
$(\id\otimes\omega_{\xi,\eta})(U) \not \in \mathbb{AP}(C_0^u(\G))$.
As $\xi\in H_c^\perp$ and $H_c^\perp$ is invariant for $U$, it follows that
\[ (\id\otimes\omega_{\xi,\eta})(U)
= (\id\otimes\omega_{\xi,\eta_0})(U)
= (\id\otimes\omega_{\xi,\eta_0})(V_0), \]
which shows that $\xi$ is not compact for $V_0$, as required.
\end{proof}

We shall now proceed to show that $V$ can be ``identified'' with a unitary
corepresentation of $\mathbb{AP}(C_0^u(G))$; part of our effort will be to
make this statement precise.

At this point, we shall need a little more of the theory of (locally)
compact quantum groups, which we now sketch.  For a compact quantum group $A$,
let $\varphi$ be
the Haar state for $A$, with GNS space $H$ and representation $\pi:A\rightarrow
\mc B(H)$.  Then, for example using the regular corepresentation of $A$ on $H$,
we can show that $\pi(A)$ also admits a coproduct making $\pi$ a morphism.
We call $\pi(A)$ the \emph{reduced form} of $A$, and with $\varphi$ as the
left and right Haar weights, this becomes a locally compact quantum group
in the sense of \cite{kv}.  Let $\K$ be the locally compact quantum group
thus induced, so $\pi(A) = C(\K)$ and $H=L^2(\K)$.  Let $L^\infty(\G)$ be the
von Neumann algebra generated by $C(\K)$ in $\mc B(L^2(\K))$; again an argument
using the regular representation shows that $L^\infty(\G)$ can be given a
weak$^*$-continuous coproduct which extends that on $C(\K)$.  Let $L^1(\K)$ be
the predual of $L^\infty(\G)$ (or, equivalently, the collection of normal
functionals on $C(\K)$) which becomes a Banach algebra for the pre-adjoint of
the coproduct on $L^\infty(\K)$.  See, for example, \cite[Section~8]{kv}.

Let $\G$ be a (locally) compact quantum group.  As the antipode $S$ is in
general unbounded, the natural way to define a $*$-structure on $L^1(\G)$
fails.  Instead, we let $L^1_\sharp(\G)$ be the collection of $\omega\in L^1(\G)$
such that there is $\omega^\sharp\in L^1(\G)$ with
\[ \ip{x}{\omega^\sharp} = \overline{ \ip{S(x)^*}{\omega} }
\qquad (x\in D(S)\subseteq L^\infty(\G)). \]
As $S(S(x)^*)^*=x$ for $x\in D(S)$, it follows that $\sharp$ gives an
involution on $L^1_\sharp(\G)$.  Under the norm $\|\omega\|_\sharp =
\max( \|\omega\|, \|\omega^\sharp\| )$ also $L^1_\sharp(\G)$ becomes a
Banach algebra.  The left-regular representation
\[ \lambda:L^1(\G)\rightarrow C_0(\hh\G); \quad \omega\mapsto
(\omega\otimes\id)(W) \]
is a homomorphism with range dense in $C_0(\hh\G)$, and it restricts
to a $*$-homomorphism on $L^1_\sharp(\G)$.  Indeed, $\omega\in L^1_\sharp(\G)$
if and only if there is $\tau\in L^1(\G)$ with $\lambda(\tau) =
\lambda(\omega)^*$, and then $\tau=\omega^\sharp$.  A ``smearing'' argument
using the scaling group shows that $L^1_\sharp(\G)$ is norm dense in $L^1(\G)$;
see \cite[Section~3]{ku} and \cite[Definition~2.3]{kvvn} and the discussion after.

Let $\K$ be the compact quantum group induced by $\mathbb{AP}(C_0^u(\G))$, so
that $C(\K)$ is the reduced form of $\mathbb{AP}(C_0^u(\G))$, and we have the
quotient map $\phi: \mathbb{AP}(C_0^u(\G)) \rightarrow C(\K)$ (as shown in
\cite{daws} this map can indeed fail to be an isomorphism).  We can then
restrict $\phi^*:C(\K)^*\rightarrow \mathbb{AP}(C_0^u(G))^*$, which is an
isometric algebra homomorphism, to $L^1(\K)$, still denoted by $\phi^*$.
For $\omega\in L^1(\K)$ let $\mu \in M(C_0^u(\G))^*$ be a Hahn-Banach extension
of $\phi^*(\omega)$, recalling that
$\mathbb{AP}(C_0^u(G)) \subseteq M(C_0^u(\G))$.
Then define $(\mu\otimes\id)(V)\in\mc B(H_c)$ in the usual way:
\[ \big( (\mu\otimes\id)(V) \xi \big| \eta \big)
= \ip{\mu}{(\id\otimes\omega_{\xi,\eta})(V)}
\qquad (\xi,\eta\in H_c). \]
As $(\id\otimes\omega_{\xi,\eta})(V) \in \mathbb{AP}(C_0^u(G))$ we hence see
that $(\mu\otimes\id)(V)$ depends only on $\omega$ and not on the particular
extension chosen.  Hence define this to be $\pi(\omega)\in\mc B(H_c)$.

\begin{lemma}
The map $\pi:L^1(\K)\rightarrow\mc B(H_c)$ is a homomorphism, and when
restricted to $L^1_\sharp(\K)$ is a $*$-homomorphism.
\end{lemma}
\begin{proof}
Let $\omega\in L^1(\K)$ and let $\mu\in M(C_0^u(\G))^*$ be an extension of
$\phi^*(\omega)$.  Treating $M(C_0^u(\G))$ as a subalgebra of $C_0^u(\G)^{**}$
we can further extend $\mu$ and then find a bounded net $(\mu_\alpha)$ in
$C_0^u(\G)^*$ with
\[ \lim_\alpha \ip{\mu_\alpha}{x} = \ip{\phi(x)}{\omega}
\qquad (x\in\mathbb{AP}(C_0(\G))). \]
It then follows that
\[ (\pi(\omega)\xi|\eta) = \ip{\phi((\id\otimes\omega_{\xi,\eta})(V))}{\omega}
= \lim_\alpha ((\mu_\alpha\otimes\id)(V)\xi|\eta)
\qquad (\xi,\eta\in H_c). \]
For $i=1,2$, let $\omega_i\in L^1(\K)$ and choose a bounded net
$(\mu^{(i)}_\alpha)$ as above.  Then, for $\xi,\eta\in H_c$,
\begin{align*} (\pi(\omega_1)\pi(\omega_2)\xi|\eta)
&= \lim_\alpha ((\mu^{(1)}_\alpha\otimes\id)(V) \pi(\omega_2)\xi|\eta)
= \lim_\alpha \lim_\beta
((\mu^{(1)}_\alpha\otimes\id)(V) (\mu^{(2)}_\beta\otimes\id)(V) \xi|\eta) \\
&= \lim_\alpha \lim_\beta \ip{\mu^{(1)}_\alpha\otimes\mu^{(2)}_\beta\otimes
\omega_{\xi,\eta}}{V_{13} V_{23}}
= \lim_\alpha \lim_\beta \ip{\mu^{(1)}_\alpha\star\mu^{(2)}_\beta\otimes
\omega_{\xi,\eta}}{V},
\end{align*}
as $V$ is a corepresentation.  Now let $x\in\mathbb{AP}(C_0^u(\G))$.  As
$\Delta(x) \in \mathbb{AP}(C_0^u(\G)) \otimes \mathbb{AP}(C_0^u(\G))$ the minimal
tensor product, we see that
\[ \lim_\alpha (\mu_\alpha^{(1)}\otimes\id)\Delta(x)
= (\phi^*(\omega_1)\otimes\id)\Delta(x), \]
with convergence in norm.  Thus
\begin{align*}
\lim_\alpha \lim_\beta \ip{\mu^{(1)}_\alpha\star\mu^{(2)}_\beta}{x}
&= \lim_\alpha \lim_\beta
  \ip{\mu^{(2)}_\beta}{(\mu_\alpha^{(1)}\otimes\id)\Delta(x)}
= \lim_\alpha \ip{\phi^*(\omega_2)}{(\mu_\alpha^{(1)}\otimes\id)\Delta(x)} \\
&= \ip{\phi^*(\omega_2)}{(\phi^*(\omega_1)\otimes\id)\Delta(x)}
= \ip{\phi^*(\omega_1)\star\phi^*(\omega_2)}{x}
= \ip{\phi^*(\omega_1\star\omega_2)}{x}.
\end{align*}
Combining this with the above, we see that
\begin{align*} (\pi(\omega_1)\pi(\omega_2)\xi|\eta)
= \ip{\phi^*(\omega_1\star\omega_2)}{(\id\otimes\omega_{\xi,\eta})(V)}
= (\pi(\omega_1\star\omega_2)\xi|\eta).
\end{align*}
So $\pi$ is a homomorphism.

As a remark, it would have been tempting to work more directly with
$\mu_i\in M(C_0(\G))^*$, but then it's not clear what meaning $\mu_1\star\mu_2$
has; the above argument is essentially a resolution of this problem.

We now show that $\pi$ is a $*$-homomorphism on $L^1_\sharp(\G)$.  We know
that $\phi \circ S_{\mathbb{AP}} \subseteq S_{\K} \circ \pi$, see
\cite[Remark~12.1]{ku}.  A similar remark applies to the inclusion
$\mathbb{AP}(C_0(\K)) \subseteq M(C_0(\G))$ which is a morphism.  So for each
$\omega\in\mc B(H_c)_*$ we have that $(\id\otimes\omega)(V) \in
D(S_{\mathbb{AP}})$ and $S_{\mathbb{AP}}((\id\otimes\omega)(V))
= (\id\otimes\omega)(V^*)$ and hence also $\phi((\id\otimes\omega)(V)) \in
D(S_{\K})$ and $S_{\K}(\phi((\id\otimes\omega)(V)))
= \phi((\id\otimes\omega)(V^*))$.  So, for $\omega\in L^1_\sharp(\K)$,
\begin{align*} (\pi(\omega^\sharp)\xi|\eta) &=
\ip{\phi((\id\otimes\omega_{\xi,\eta})(V))}{\omega^\sharp}
= \overline{ \ip{S_\K(\phi((\id\otimes\omega_{\xi,\eta})(V)))^*}{\omega} }
= \overline{ \ip{\phi((\id\otimes\omega_{\eta,\xi})(V))}{\omega} } \\
&= \overline{ (\pi(\omega)\eta|\xi) }
= (\pi(\omega)^*\xi|\eta),
\end{align*}
as required.
\end{proof}

It is not clear to us that $\pi$ is non-degenerate, but see
Corollary~\ref{corr:is_nondeg} below.  However, let $H_{cc}$ be
the closed linear span of $\{ \pi(\omega)\xi : \xi\in H_c,
\omega\in L^1(\K) \}$, and note that if we replace $L^1(\K)$ by
$L^1_\sharp(\K)$ we get the same closure.  Then, for $\eta\in H_{cc}^\perp$
we find that
\[ (\pi(\omega)\eta|\xi) = (\eta|\pi(\omega^\sharp)\xi) = 0
\qquad (\xi\in H_c, \omega\in L^1_\sharp(\K)), \]
and so $\pi$ restricts to the zero representation on $H_{cc}^\perp$.
As products are dense in $L^1(\K)$ it follows that the closed linear span of
$\{ \pi(\omega)\xi: \omega\in L^1(\K), \xi\in H_{cc} \}$ equals $H_{cc}$ again,
and so $\pi$, restricted to $H_{cc}$, is non-degenerate, say giving $\pi_{cc}$.

\begin{proposition}
There is a unitary corepresentation $Y$ of $\mathbb{AP}(C_0^u(\G))$ on
$H_{cc}$ such that, if we zero-extend $Y$ to a degenerate representation
on $H_c$, then $(\id\otimes\omega)(Y) = (\id\otimes\omega)(V) \in M(C_0^u(\G))$
for all $\omega\in\mc B(H_c)_*$.
\end{proposition}
\begin{proof}
As $\pi_{cc}$ is a non-degenerate $*$-homomorphism of $L^1_\sharp(\K)$,
by \cite[Corollary~4.3]{ku}, there is a unitary corepresentation
$X\in M(C(\K)\otimes \mc B_0(H_{cc}))$ with $(\omega\otimes\id)(X) =
\pi_{cc}(\omega)$ for $\omega\in L^1(\K)$.  Treating $\mc B(H_{cc})$ as a
``corner'' of $\mc B(H_c)$, we may ``zero-extend'' $X$ to a member of
$M(C(\K)\otimes \mc B_0(H_{c}))$.  A careful check shows that
$\pi(\omega) = (\omega\otimes\id)(X)$ for all $\omega\in L^1(\mathbb K)$.
Then, for $\tau\in\mc B(H_c)_*$, we have that
\[ \ip{(\id\otimes\tau)(X)}{\omega} = \ip{\pi(\omega)}{\tau}
= \ip{\phi^*(\omega)}{(\id\otimes\tau)(V)}
= \ip{\phi( (\id\otimes\tau)(V) )}{\omega}\qquad (\omega\in L^1(\K)). \]
So $(\id\otimes\tau)(X) = \phi( (\id\otimes\tau)(V) )$.  

Using \cite[Proposition~6.6]{ku} we may ``lift''
$X\in M(C(\K)\otimes\mc B_0(H_{cc}))$ to a unique unitary corepresentation
$X_u\in M(C^u(\K)\otimes\mc B_0(H_{cc}))$, and then push this down to
a (unique) unitary corepresentation $Y\in M(\mathbb{AP}(C_0^u(\G)) \otimes
\mc B_0(H_{cc}))$ with $(\phi\otimes\id)(Y) = X$.
Let $W\in M(\mathbb{AP}(C_0^u(\G))\otimes L^2(\K))$ be the left regular
representation of $\mathbb{AP}(C_0^u(\G))$, so that
\[ W^*(1\otimes\phi(x)) W = (\id\otimes\phi)\Delta_{\mathbb{AP}}(x)
\qquad (x\in \mathbb{AP}(C_0^u(\G))). \]
(This follows from equation (5.10) in \cite{woro2}, or from \cite{ku}, compare
the proof of \cite[Lemma~3.6]{daws}).  As $Y$ is a corepresentation,
it follows that
\[ W^*_{12} (\phi\otimes\id)(Y)_{23} W_{12} = (\id\otimes\phi\otimes\id)
(Y_{13} Y_{23}) \implies
Y_{13} = W^*_{12} X_{23} W_{12} X_{23}^*, \]
as $(\phi\otimes\id)(Y) = X$.  This of course being the proof that $Y$
is uniquely determined by $X$.  If we now zero-extend $Y$ to a member of
$M(\mathbb{AP}(C_0(\G)) \otimes \mc B_0(H_c))$, then again
$W^*_{12} X_{23} W_{12} = Y_{13} X_{23}$ and so
$W^*_{12} X_{23} W_{12}X_{23}^* = Y_{13} (1\otimes 1\otimes p) = Y_{13}$
as $Y_{13}$ is zero on $H_{cc}^\perp$ by construction, where here
$p:H_c\rightarrow H_{cc}$ is the orthogonal projection.

Now let $\xi,\eta\in H_c$, so
\begin{align*}
(\id\otimes\omega_{\xi,\eta})(Y) \otimes 1
= (\id\otimes\id\otimes\omega_{\xi,\eta})(W^*_{12} X_{23} W_{12} X_{23}^*)
= W^* \sum_i (\id\otimes\id\otimes\omega_{\xi,\eta})(X_{23}
  (W\otimes\theta_{e_i,e_i}) X_{23}^*),
\end{align*}
where $(e_i)$ is an orthonormal basis for $H_c$ and $\theta_{e_i,e_i}$ is
the rank-one projection onto the span of $e_i$.  Notice that
$\sum_i (W\otimes\theta_{e_i,e_i})$ converges strictly to $W\otimes 1$
in $M(\mathbb{AP}(C_0^u(\G))\otimes \mc B_0(L^2(\K))\otimes\mc B_0(H_{cc}))$,
and as slice maps are strictly continuous, the sum above converges strictly
in $M(\mathbb{AP}(C_0^u(\G)) \otimes C_0(\K))$.  Thus
\[ (\id\otimes\omega_{\xi,\eta})(Y) \otimes 1
= W^* \sum_i ( 1\otimes(\id\otimes\omega_{e_i,\eta})(X)) W
( 1\otimes(\id\otimes\omega_{\xi,e_i})(X^*)). \]
However, this sum is then equal to
\begin{align*} W^* \sum_i ( 1\otimes & \phi(\id\otimes\omega_{e_i,\eta})(V)) W
  ( 1\otimes\phi(\id\otimes\omega_{\xi,e_i})(V^*)) \\
&= \sum_i (\id\otimes\phi)\Delta_{\mathbb{AP}}((\id\otimes\omega_{e_i,\eta})(V))
  ( 1\otimes\phi(\id\otimes\omega_{\xi,e_i})(V^*)) \\
&= (\id\otimes\phi)\Big( \sum_i
  \Delta_{\mathbb{AP}}((\id\otimes\omega_{e_i,\eta})(V))
  ( 1\otimes\phi(\id\otimes\omega_{\xi,e_i})(V^*)) \Big) \\
&= (\id\otimes\phi)\Big( \sum_i
  (\id\otimes\id\otimes\omega_{e_i,\eta})(V_{13}V_{23}	)
  ( 1\otimes\phi(\id\otimes\omega_{\xi,e_i})(V^*)) \Big) \\
&= (\id\otimes\phi)\Big( \sum_i
  (\id\otimes\id\otimes\omega_{\xi,\eta})(V_{13}V_{23}
   (1\otimes 1\otimes \theta_{e_i,e_i})V^*_{23}) \Big) \\
&= (\id\otimes\phi)\big( (\id\otimes\id\otimes\omega_{\xi,\eta})(V_{13}) \big)
= (\id\otimes\omega_{\xi,\eta})(V) \otimes 1.
\end{align*}
In the final part, of course $V$ is unitary on all of $H_c$, not just $H_{cc}$.
Thus we have shown that $(\id\otimes\omega)(Y) = (\id\otimes\omega)(V)$ for
all $\omega\in\mc B(H_c)_*$.
\end{proof}

Consider the (non-degenerate) inclusions
\[ \xymatrix{ \mathbb{AP}(C_0^u(\G))\otimes\mc B_0(H_c)
\ar[r] & M(C_0^u(\G)) \otimes \mc B_0(H_c) \ar[r] &
M(C_0^u(\G) \otimes \mc B_0(H_c)). } \]
Call the composition $\pi$, so we get the strict extension $\tilde\pi:
M( \mathbb{AP}(C_0(\G))\otimes\mc B_0(H_c) ) \rightarrow
M(C_0(\G) \otimes \mc B_0(H_c))$.  This is injective, as if $\tilde\pi(x)=0$
then $\pi(xa)=0$ for all $a\in \mathbb{AP}(C_0^u(\G))\otimes\mc B_0(H_c)$
and as $\pi$ is an inclusion, hence $xa=0$ for all $a$, so $x=0$.

\begin{proposition}
With the above notation, we have that $\tilde\pi(Y) = V$.
\end{proposition}
\begin{proof}
Let $x\in\mathbb{AP}(C_0^u(\G)), a\in C_0^u(\G)$ and $\theta_1,\theta_2\in
\mc B_0(H_c)$.  We need to show that
\[ \tilde\pi(Y) \pi(x\otimes\theta_1) (a\otimes\theta_2)
= \pi\big( Y(x\otimes\theta_1) \big) (a\otimes\theta_2)
\quad\text{and}\quad
V\pi(x\otimes\theta_1) (a\otimes\theta_2)
= V(xa\otimes\theta_1\theta_2), \]
agree; similarly there is an analogous argument ``on the left''.
Now, both these are members of $C_0^u(\G)\otimes\mc B_0(H_c)$, so we can
prove equality by slicing by $\omega\in\mc B(H_c)_*$.  However,
\begin{align*}
(\id\otimes\omega) \big( \pi\big( Y(x\otimes\theta_1) \big)
  (a\otimes\theta_2) \big)
&= (\id\otimes\theta_2\omega)\big( \pi\big( Y(x\otimes\theta_1) \big)\big) a
= (\id\otimes\theta_2\omega)\big( Y(x\otimes\theta_1) \big) a \\
&= \big((\id\otimes\theta_1\theta_2\omega)(Y) x\big) a
= \big((\id\otimes\theta_1\theta_2\omega)(V) x\big) a \\
&= (\id\otimes\omega)\big( V(xa\otimes\theta_1\theta_2) \big),
\end{align*}
as required.
\end{proof}

\begin{corollary}\label{corr:is_nondeg}
$\pi$ is non-degenerate: $H_{cc}=H_c$.
\end{corollary}

Let $\iota:\mathbb{AP}(C_0^u(\G)) \rightarrow M(C_0^u(\G))$ be the inclusion,
which is a quantum group morphism.  For any corepresentation $Z\in
M(\mathbb{AP}(C_0^u(\G)) \otimes \mc B_0(K))$, the canonical way to induce a
corepresentation of $C_0^u(\G)$ is by looking at $(\iota\otimes\id)(Z)$.
If we unpack this, then $\iota\otimes\id:M(\mathbb{AP}(C_0^u(\G)) \otimes
\mc B_0(K)) \rightarrow M(C_0^u(\G)\otimes\mc B_0(K))$ is precisely the map
$\pi$ above.

So, in conclusion, $V$ arises from a unitary corepresentation, $Y$, of the
compact quantum group $\mathbb{AP}(C_0^u(\G))$, in the canonical way.
As $Y$ decomposes as a direct sum of finite-dimensional corepresentations,
the same is hence true of $V$.  As $\mathbb{AP}(C_0^u(\G))$ is compact, these
are also automatically \emph{admissible} in So{\l}tan's sense.
So we have proved the following.

\begin{theorem}\label{thm:subreps_cmptvecs}
Let $U$ be a unitary corepresentation of $\G$.  Then $U$ has a 
finite-dimensional, admissible sub-corepresentations if and only if $U$ has
non-zero compact vectors.
\end{theorem}

\section{Decompositions}\label{sec:decomp}

For a locally compact quantum group, let $S$ be the antipode, and $R$ the
unitary antipode, here supposed to act on $C_0^u(\G)$, see
\cite[Sections~7 and~9]{ku}.
In this section we restrict to Kac algebras, or slightly more generally,
to those locally compact quantum groups where the
antipode and the unitary antipode agree.  In this case, the universal
corepresentation $\mc W \in M(C_0^u(\G) \otimes C_0^u(\hh\G))$ generates a
C$^*$-Eberelin algebra.  Indeed, the coproduct structure on the dual quantum
group $C_0^u(\hh\G)$ can be used to show that
$B_0 = \{ (\id\otimes\hh\mu)(\mc W): \hh\mu\in C_0^u(\hh\G)^* \}$ is an
algebra.  As $R=S$, from the usual formula (see \cite[Remark~9.1]{ku}) we get
\[ R\big( (\id\otimes\hh\mu)(\mc W) \big) = (\id\otimes\hh\mu)(\mc W^*). \]
For any corepresentation $U$ let $U^c = (R\otimes\top)(U)$ be the
\emph{contragradient} corepresentation (see \cite[Proposition~10]{sw}),
recalling the tranpose map $\top$ from Section~\ref{sec:note}.
Then both $R$ and $\top$
are anti-$*$-homomorphisms, and so $(R\otimes\top)(U)$ is well-defined.
Furthermore, we see that
\[ (\id\otimes\omega_{\overline\xi,\overline\eta})(U^c) =
R\big( (\id\otimes\omega_{\eta,\xi})(U) \big)
= (\id\otimes\omega_{\eta,\xi})(U^*)
= (\id\otimes\omega_{\xi,\eta})(U)^*. \]
So slices of $U^c$ give the adjoints of slices of $U$.  As $\mc W$ is the
universal corepresentation, slices of $\mc W^c$ form a subset of slices
of $\mc W$.  It follows that $B_0^* \subseteq B_0$ and hence actually $B_0$
is $*$-closed, and hence its closure $B$ is a C$^*$-algebra.  Again by
universality, clearly $E(\G) = B$.

Now let $U$ be a corepresentation of $C_0^u(\G)$ on $H$.  That $\mc W$
is universal means that $(\id\otimes\omega)(U) \in E(\G)$ for any
$\omega\in \mc B(H)_*$.  With $M$ the (unique) invariant mean on $E(\G)$
it thus makes sense to define $p=(M\otimes\id)(U)$, as by definition
$(p\xi|\eta) = \ip{M}{(\id\otimes\omega_{\xi,\eta})(U)}$ for $\xi,\eta\in H$.
We can now copy the start of the proof of Theorem~\ref{thm:invmeanuni}, carefully
noting that we don't need $U$ to generate a C$^*$-Eberlein algebra, only
that slices of $U$ are in $E(\G)$.  Thus we may conclude that $p$ is the
orthogonal projection onto $\inv(U)$ and so
$\ip{M}{(\id\otimes\omega_{\xi,\eta})(U)} 1
= (\id\otimes\omega_{p\xi,\eta})(U)$.

\begin{lemma}\label{lem:mean_inv_vecs}
Let $U$ be a corepresentation of $C_0^u(\G)$ on $H$, and consider the
corepresentation $U^c \tp U$ on $\overline{H}\otimes H$.  Let $p$ be the
orthogonal projection onto $\inv(U^c \tp U)$.
For $x=(\id\otimes\omega_{\xi,\xi})(U)$ we have that $\ip{M}{x^*x}=0$
if and only if $p(\overline\xi \otimes \xi)=0$.
\end{lemma}
\begin{proof}
We see that
\[ x^*x = (\id\otimes\omega_{\overline\xi,\overline\xi})(U^c)
(\id\otimes\omega_{\xi,\xi})(U)
= (\id\otimes\omega_{\overline\xi\otimes\xi,\overline\xi\otimes\xi})
  (U^c\tp U). \]
and so $\ip{M}{x^*x} 1 = (\id\otimes\omega_{p(\overline\xi\otimes\xi),
\overline\xi\otimes\xi})(U^c\tp U)$.  For any state $\mu\in C_0^u(\G)^*$,
as $p$ is the projection onto $\inv(U^c\tp U)$, it follows that
$(\mu\otimes\id)(U^c\tp U)p=p$ and so
\[ M(x^*x) = \ip{\mu}{(\id\otimes\omega_{p(\overline\xi\otimes\xi),
\overline\xi\otimes\xi})(U^c\tp U)}
= (p(\overline\xi\otimes\xi)|\overline\xi\otimes\xi)
= \| p(\overline\xi\otimes\xi) \|^2, \]
from which the result follows.
\end{proof}

We remark that the Cauchy-Schwarz inequality, and polar decomposition, show
easily that $\ip{M}{x^*x}=0$ if and only if $\ip{M}{|x|}=0$.

Motivated by the above, we wish to study when $\inv(U^c\tp U)$ is non-trivial.

\begin{proposition}\label{prop:fdsubcorep}
For any corepresentation $U$, we have that $\inv(U^c\tp U)\not=\{0\}$
if and only if $U$ has a finite-dimensional sub-corepresentation.
\end{proposition}
\begin{proof}
This is stated for discrete Kac algebras in \cite[Theorem~2.6]{ks}, but
the proof readily generalises.
\end{proof}

The following is a generalisation of Godement's decomposition theorem,
\cite[Theorem~16, page~64]{G}.

\begin{theorem}\label{thm:decomp}
Let $U$ be a corepresentation of $C_0^u(\G)$.  As in Section~\ref{sec:cmptvecs},
$U = U_c \oplus U_0$ on $H=H_c\oplus H_c^\perp$ where $H_c$ is the subspace of
compact vectors.  Then $U_c$ is canonically induced by a corepresentation
of $\mathbb{AP}(C_0^u(\G))$, and for any $x=(\id\otimes\omega)(U_0)$, we have
that $\ip{M}{x^*x}=0$.
\end{theorem}
\begin{proof}
Given the results of Section~\ref{sec:cmptvecs} all we have to show is that
$\ip{M}{x^*x}=0$ for $x=(\id\otimes\omega)(U_0)$.
By Lemma~\ref{lem:rest_com_vecs}, we know that $U_0$ has no non-zero compact
vectors.  By replacing $H_c^\perp$ with $H_c^\perp\otimes\ell^2$ we may suppose
that $\omega=\omega_{\xi,\eta}$; notice that $H_c^\perp\otimes\ell^2$ also has
no non-zero compact vectors, as if $\sum_i \xi_i\otimes \delta_i$ is compact
then slicing against $\eta\otimes\delta_j$ for a fixed $j$ shows that
$\xi_j$ is compact, so $\xi_j=0$.

As in the proof of Lemma~\ref{lem:mean_inv_vecs}, we look at
\[ \ip{M}{x^*x} = \ip{M}{(\id\otimes\omega_{\overline\xi\otimes\xi,
\overline\eta\otimes\eta})(U_0^c\tp U_0)}
= \ip{\mu \otimes \omega_{p(\overline\xi\otimes\xi),
\overline\eta\otimes\eta}}{U_0^c\tp U_0} \]
which holds for any state $\mu$, with again $p$ the orthogonal projection
onto $\inv(U_0^c\tp U_0)$.  As $U_0$ has no non-zero compact vectors,
Theorem~\ref{thm:subreps_cmptvecs} shows that $U_0$ has no finite-dimensional
(admissible) sub-corepresentations, and so Proposition~\ref{prop:fdsubcorep} shows
that $\inv(U_0^c\tp U_0)=\{0\}$.  Thus $p=0$ and so $\ip{M}{x^*x}=0$ as claimed.
\end{proof}

A similar result is the following.

\begin{proposition}\label{prop:no_cmpt_vecs}
Let $U$ be a corepresentation of $\G$.  Then $U$ has no compact vectors
if and only if $\ip{M}{|x|}=0$ for all $x=(\id\otimes\omega)(U)$.
\end{proposition}
\begin{proof}
As $\ip{M}{x^*x}=0$ if and only if $\ip{M}{|x|}=0$, given the theorem, we
need only show that if $U$ has compact vectors, then for some choice of
$x=(\id\otimes\omega)(U)$ we have that $\ip{M}{x^*x}\not=0$.
By Theorem~\ref{thm:subreps_cmptvecs}, in this case $U$ admits a
finite-dimensional sub-corepresentation, say
\[ V = \sum_{i,j=1}^n v_{i,j} \otimes e_{i,j}, \]
where $U$ acts on $H$, where $(e_i)_{i=1}^n$ is an orthonormal subset of $H$,
and where $(e_{i,j})$ are the matrix units associated with $(e_i)$.  So
$e_{i,j}$ is the rank-one operator sending $e_j$ to $e_i$.

That $V$ is unitary means that $\sum_i v_{i,j}^* v_{i,j} = 1$ for all $j$, and
so $1 = \ip{M}{1} = \sum_i \ip{M}{v_{i,j}^* v_{i,j}}$.  Thus there is some
choice of $i,j$ with $\ip{M}{v_{i,j}^* v_{i,j}}\not=0$.
Then we compute that $(\id\otimes\omega_{e_j,e_i})(U) = v_{i,j}$, and so
$x = (\id\otimes\omega_{e_j,e_i})(U) = v_{i,j}$ is a suitable choice to
give $\ip{M}{x^*x}\not=0$ as required.
\end{proof}

\subsection{At the algebra level}\label{sec:alg_level}

For a locally compact group $G$, we have that $\mc E(G) = AP(C_0(G)) \oplus
\mc E_0(G)$, this usually being proved by compact semigroup considerations,
for example \cite[Chapter 4]{bjm}.
In this section, we explore this for Kac algebras, and in particular, give
a direct $C^*$-algebraic proof in the group case.

\begin{proposition}
Let $\G$ be a Kac algebra and let $M$ be the invariant mean on $\mc E(\G)$.
Then $M$ is a trace.
\end{proposition}
\begin{proof}
Let $U$ be a corepresentation of $\G$ on $H$, so by Theorem~\ref{thm:decomp},
we have that $U = U_c \oplus U_0$ with respect to $H=H_c\oplus H_0$.  Let
$\xi=\xi_c+\xi_0, \eta=\eta_c+\eta_0$ with respect to this orthogonal
decomposition, and set
\[ x = (\id\otimes\omega_{\xi,\eta})(U) =
(\id\otimes\omega_{\xi_c,\eta_c})(U_c)
+ (\id\otimes\omega_{\xi_0,\eta_0})(U_0)
= x_c + x_0 \]
say, where $x_c \in \mathbb{AP}(C_0^u(\G))$ and $\ip{M}{x_0^*x_0}=0$.
By Cauchy-Schwarz, also $\ip{M}{x_0^*x_c} = \ip{M}{x_c^*x_0}=0$ and so
$\ip{M}{x^*x} = \ip{M}{x_c^*x_c}$.  As $x^*$ is a slice of $U^c$,
by applying the same argument to $U^c$, we conclude also that
$\ip{M}{xx^*} = \ip{M}{x_cx_c^*}$.

The inclusion $\mathbb{AP}(C_0^u(\G)) \rightarrow
M(C_0^u(\G))$ intertwines the coproducts.  Let $(\tau_t)$ be the scaling group
of $\mathbb{AP}(C_0^u(\G))$, so by \cite[Remark~12.1]{ku}, the inclusion
intertwines $\tau_t$ and the scaling group of $M(C_0^u(\G))$, which is of
course trivial as $\G$ is Kac.  It follows that $\tau_t=\id$ for all $t$,
and so $\mathbb{AP}(C_0^u(\G))$ is a compact Kac algebra.

By Proposition~\ref{prop:haarstate}, $M$ restricted to $\mathbb{AP}(C_0^u(\G))$
is the Haar state, and so is a trace on
$\mathbb{AP}(C_0^u(\G))$.  Hence $\ip{M}{x^*x} = \ip{M}{x_c^*x_c}
= \ip{M}{x_cx_c^*} = \ip{M}{xx^*}$.  As such $x$ are dense in $\mc E(\G)$ the
result follows by continuity.
\end{proof}

\begin{proposition}\label{prop:quot_onto_red}
Let $\G$ be a Kac algebra, let $M$ be the invariant mean on $\mc E(\G)$, and
let $\mc E_0(\G) = \{ x\in\mc E(\G) : \ip{M}{x^*x}=0 \}$.
Then $\mc E_0(\G)$ is a closed two-sided ideal, and with $(H,\pi,\xi_0)$
the GNS triple of $M$, we have that $\mc E(\G)/\mc E_0(\G) = \pi(\mc E(\G))$. Furthermore, if $\mathbb{AP}(C_0^u(\G))$
induces the compact quantum group $\G^{\sap}$ then $H$ is isomorphic to
$L^2(\G^{\sap})$ in such a way that $\pi(\mc E(\G)) = C(\G^\sap)$.
\end{proposition}
\begin{proof}
As $M$ is a trace, it follows that $\mc E_0(\G)$ is a two-sided ideal, the
kernel of $\pi$ is $\mc E_0(\G)$, and that hence $\pi(\mc E(\G)) =
\mc E(\G) / \mc E_0(\G)$.  By the proof of the
previous proposition, we also immediately see that $\pi(\mc E(\G)) = 
\pi(\mathbb{AP}(C_0^u(\G)))$.  Now, $L^2(\G^\sap)$ is the GNS space for the
Haar state on $\mathbb{AP}(C_0^u(\G))$, which is now seen to be isomorphic to
$H$, and $C(\G^\sap) = \pi(\mathbb{AP}(C_0^u(\G)))$ which we have just shown
to be equal to $\pi(\mc E(\G)) = \mc E(\G) / \mc E_0(\G)$.
\end{proof}

We hence obtain a short exact sequence of $C^*$-algebras
\[ 0 \rightarrow \mc E_0(\G) \rightarrow \mc E(\G) \rightarrow
C(\G^\sap) \rightarrow 0 \]
where the first map in inclusion, and the second map is $\pi$.

\begin{theorem}
If $\mathbb{AP}(C_0^u(\G))$ is reduced, in particular, if $\G^\sap$ is
coamenable, then $\mc E(\G) = \mc E_0(\G) \oplus \mathbb{AP}(C_0^u(\G))$.
\end{theorem}
\begin{proof}
By definition, if $\mathbb{AP}(C_0^u(\G))$ is reduced then
$\pi:\mathbb{AP}(C_0^u(\G)) \rightarrow C(\G^\sap)$ is an isomorphism
and so there is a map $\iota:C(\G^\sap) \rightarrow \mathbb{AP}(C_0^u(\G))$
with $\pi\circ\iota = \id$.  Treating $\iota$ as a map $C(\G^\sap)
\rightarrow \mc E(\G)$, we get a splitting of the short exact sequence,
and so $\mc E(\G)$ splits as
$\mc E_0(\G) \oplus \mathbb{AP}(C_0^u(\G))$.  If $\G^\sap$ is coamenable
then, compare \cite[Theorem~3.1]{bt}, $\mathbb{AP}(C_0^u(\G))$ is
necessarily reduced.
\end{proof}

This result immediately gives us the claimed result in the group case,
as a group is always co-amenable.  We shall see below in
Section~\ref{sec:cocomm} that even in the cocommutative case, we need some
condition to get such a splitting result.

\section{Compactification examples}

We motivated our approach to compactifications in Section~\ref{sec:cmpts}
by showing that it agreed with the ideas of Spronk and Stokke in \cite{ss}.
Indeed, \cite[Theorem~3.14]{ss} gives a construction of $G^{\mc CH}$
by showing that $G^{\mc CH}$ can be realised
as the \emph{Eberlein compactification} of $G$, constructed in
\cite[Section~3.1]{ss}.  Namely, let $\omega_G:G\rightarrow
\mc{U}(H_{\omega_G})$ be the universal (unitary, SOT continuous) representation
of $G$, and then set $G^{\mc CH}$ to be the weak$^*$-closure of $\omega_G(G)$;
notice that by construction $G^{\mc CH}\in\wch$.

This is the same as our construction, as Theorem~\ref{thm:comp_construction}
above essentially constructs a maximal corepresentation $U$ of $C_0(G)$
with the property that slices of $U$ are dense in a $*$-algebra.
In Section~\ref{sec:decomp} above we saw that for Kac algebras, the
universal corepresentation has this property.  As corepresentations of
$C_0(G)$ biject with representations of $G$, this is the same construction
as in \cite{ss}.

\begin{remark}\label{rem:haa_no_good}
In Section~\ref{sec:haatenprod} we showed links with the extended Haagerup
tensor product.  Suppose $G$ is a locally compact group.  It would be tempting
to look at functions $F\in L^\infty(G)$ with $\Delta(F)$ in the norm closure of
$L^\infty(G) \otimes^{eh} L^\infty(G)$, say giving the collection $E_{eh}(G)$.
This would give a C$^*$-algebra which contains $\mc E(G)$ (by the remarks in 
Section~\ref{sec:haatenprod}) and seems to have many of the properties of
$\mc E(G)$.  However, it seems quite possible that this gives something larger
than $\mc E(G)$, even in this classical (i.e. $G$ is a group) setting.

We follow Spronk's paper \cite{spronk}.
Now, $\Delta(F)$ is the function $G\times G\rightarrow\mathbb C;
(s,t)\mapsto F(st)$.  As $L^\infty(G)$ is commutative, it is not hard to
see that the function $u(s,t) = \Delta(F)(s,t^{-1}) = F(st^{-1})$ will then
also be in $L^\infty(G)\otimes^{eh} L^\infty(G)$.  Furthermore, $u$ is
\emph{invariant} in the sense of \cite[Section~5]{spronk} as $u(sr,t)
= F(srt^{-1}) = F(s(tr^{-1})^{-1}) = u(s,tr^{-1})$.  Thus $a:G\rightarrow
\mathbb C$ defined by $a(st)=u(s,t^{-1}) = F(st)$ is a \emph{completely
bounded multiplier} of $A(G)$; but $a=F$ of course.  We can reverse this
argument (see again \cite{spronk}) and so we conclude that $E_{eh}(G)$ is
precisely the norm closure of $M_{cb}A(G)$ in $L^\infty(G)$.  However, it
is not clear to us why, for arbitrary $G$, this wouldn't be larger than
$\mc E(G)$.
\end{remark}

\subsection{The cocommutative case}\label{sec:cocomm}

We now study the cocommutative case.  Let $G$ be a locally compact group, and
set $\mathbb G=\hat G$  so that
$C^u_0(\mathbb G) = C^*(G)$ the (full) group C$^*$-algebra of $G$.  Again,
this is a Kac algebra, and so $\mc E(G)$ is the closure of the image of the measure
algebra $M(G) = C_0(G)^*$ inside $M(C^*(G))$.  We also know that
$\mathbb{AP}(C^*(G))$ is the norm closure of the translation operators
$\{ \lambda(g) : g\in G \}$ inside $M(C^*(G))$, see \cite[Proposition~4.3]{soltan}
or \cite[Section~5]{daws}.

Let us think about compact vectors, and decompositions, in this setting.
We shall just sketch these results, leaving the details to be verified by the
reader.  There is a bijection between corepresentations of $C^*(G)$ and
$*$-representations of $C_0(G)$.  Any $*$-representation of $C_0(G)$ is the 
direct sum of cyclic representations, and these all have the form of $C_0(G)$ 
acting by multiplication on $L^2(G,\mu)$, where $\mu$ is a Borel probability 
measure on $G$.  Let $\pi:C_0(G)\rightarrow L^2(G,\mu)$ induce the
corepresentation $U\in M(C^*(G)\otimes\mc B(L^2(G,\mu)))$.  For
$\xi,\eta\in L^2(G,\mu)$ the coefficient $(\id\otimes\omega_{\xi,\eta})(U)
\in \mc E(\hat G) \subseteq M(C^*(G))$ is the operator given by (left)
convolution by the measure $\phi$ where
\[ \int_G f(s) \ d\phi(s) = \int_G f(s^{-1}) \xi(s) \overline{\eta(s)}
 \ d\mu(s). \]
The $s^{-1}$ occurs because of our choice of conventions: essentially
because $\widehat W = \sigma W^*\sigma$.

We know that $\mathbb{AP}(C^*(G))$ is equal to the closed linear span of the
translation operators $\{ \lambda(s):s\in G\}$ in $M(C^*(G))$; see
\cite[Section~4.2]{soltan}.  This is a compact quantum group living between
$C^*(G_d)$ and $C^*_r(G_d)$, where $G_d$ is $G$ with the discrete topology;
compare \cite[Section~6.2]{daws}.  Let $U$ be a corepresentation of $\mathbb{AP}(C^*(G))$ (which always comes from a corepresentation of
$C^*(G_d)$, so from a $*$-representation $\pi_d$ of $c_0(G_d)$) and let this
induce a corepresentation of $C^*(G)$, so a $*$-representation $\pi$ of
$C_0(G)$ on $H$.  Then for $f\in C_0(G)$, let this give a function in
$M(C_0(G_d)) \cong \ell^\infty(G)$, and then $\pi(f) = \pi_d(f)$ where
we take the strict extension of $\pi_d$ to $\ell^\infty(G)$.  It follows that
$\pi$ is equivalent to a sum of ``multiplication'' representations on
$L^2(G,\mu)$ where $\mu$ is a purely atomic measure (because this is how
$\pi_d$ will arise).

So given $\pi:C_0(G)\rightarrow\mc B(L^2(G,\mu))$ inducing $U$, the
decomposition result of Theorem~\ref{thm:decomp} induces a decomposition of
$\pi$, say as $\pi=\pi_c+\pi_0$.  This comes from decomposing $\mu=\mu_c+\mu_0$
where $\mu_c$ is the atomic part of $\mu$, and $\mu_0$ is the non-atomic part,
and where $\pi_c:C_0(G)\rightarrow\mc B(L^2(G,\mu_0))$, and similarly for
$\pi_0$.  Then the results of Section~\ref{sec:alg_level} readily show that
if $\mu$ is a measure on $G$, inducing the convolution operator
$L_\mu\in M(C^*(G))$, then the invariant mean satisfies $\ip{M}{L_\mu} =
\mu( \{e\} )$, so singling out the atomic part at the identity of $e$.

For example, if $G_d$ is not amenable, then the canonical map from
$\mathbb{AP}(C_0^u(\hat G)) \rightarrow C^*_r(G_d)$ is not an isomorphism,
see \cite[Section~6.2]{daws}.  Thus $M$, restricted to 
$\mathbb{AP}(C_0^u(\hat G))$, is not faithful, and so $\mc E_0(\G) \cap
\mathbb{AP}(C_0^u(\hat G))$ is more than just $\{0\}$.  The ``problem'' here
is that $\mathbb{AP}(C_0^u(\hat G))$ is larger than $C^*_r(G_d)$.  However,
there can be no map $r:C^*_r(G_d) \rightarrow \mathbb{AP}(C_0^u(\hat G))$
with $q\circ r = \id$, where $q:\mc E(\hat G) \rightarrow C^*_r(G_d)$ is
as in Proposition~\ref{prop:quot_onto_red}.  Indeed, for any $s\in G$, we have
that $q(\lambda(s)) = \lambda(s) \in C^*_r(G_d)$, and so we would need that
$r(\lambda(s)) = \lambda(s)$.  As both $C^*_r(G_d)$ and
$\mathbb{AP}(C_0^u(\hat G))$ are the closure of $\{\lambda(s):s\in G\}$ in
different algebras, this would force $r$ to be have dense range, and have
closed range (as a $*$-homomorphism) and hence be onto, a contradiction.

In the literature, it is more usual to work with $C^*_r(G)$, instead of $C^*(G)$,
and hence view $M(C^*_r(G))$ as a subalgebra of $VN(G)$.  In this context, it is
well-known that even $\wap(VN(G))$ always has a unique invariant mean,
\cite[Proposition~5]{gran}.  Ideas somewhat related to, but put to different
ends than, those in this section can be found in \cite{b, dr}.

\vspace{5ex}

\noindent\emph{Authors's Addresses:}

\noindent\parbox[t]{3in}{
Instytut Matematyczny PAN\\
8, ul. \'Sniadeckich\\
00-656, Warsaw\\
Poland

\smallskip\noindent\emph{Email:} \texttt{biswarupnow@gmail.com}}
\hspace{2em}
\parbox[t]{3in}{School of Mathematics\\
University of Leeds\\
Leeds LS2 9JT\\
United Kingdom

\smallskip\noindent\emph{Email:} \texttt{matt.daws@cantab.net}}

\end{document}